\documentclass[12pt, reqno]{amsart}
\usepackage{amssymb,amsthm,amsfonts,amsmath, amscd}

\usepackage{hyperref}
\usepackage{mathrsfs}
\usepackage{comment}

\newcommand\Y{\mathbb Y}
\newcommand\Z{\mathbb Z}
\newcommand\C{\mathbb C}
\newcommand\R{\mathbb R}

\newcommand\E{\mathbb E}
\newcommand\N{\mathbb N}

\newcommand\F{\mathbb F}
\newcommand\ASC{\operatorname{ASC}}

\newcommand\al{\alpha}
\newcommand\be{\beta}

\newcommand\De{\Delta}
\newcommand\de{\delta}

\newcommand\la{\lambda}
\newcommand\si{\sigma}

\newcommand\eps{\varepsilon}

\newcommand\om{\omega}
\newcommand\Om{\Omega}
\newcommand\z{\zeta}

\newcommand\wt{\widetilde}

\newcommand\const{\operatorname{const}}

\newcommand\Sym{\operatorname{Sym}}

\newcommand\ccdot{\,\cdot\,}

\newtheorem{theorem}{Theorem}[section]
\newtheorem{proposition}[theorem] {Proposition}

\newtheorem{corollary}[theorem]{Corollary}
\newtheorem{lemma}[theorem]{Lemma}

\theoremstyle{definition}
\newtheorem{definition}[theorem]{Definition}
\newtheorem{remark}[theorem]{Remark}

\numberwithin{equation}{section}

\begin{document}

\title[]{Elements of the $q$-Askey scheme in the algebra of symmetric functions}

\author{Cesar Cuenca and Grigori Olshanski}

\begin{abstract}
The classical $q$-hypergeometric orthogonal polynomials are assembled into a hierarchy called the $q$-Askey scheme. At the top of the hierarchy, there are two closely related families, the Askey-Wilson and $q$-Racah polynomials. As it is well known, their construction admits a generalization leading to remarkable orthogonal symmetric polynomials in several variables.

We construct an analogue of the multivariable $q$-Racah polynomials in the algebra of symmetric functions. Next, we show that our $q$-Racah symmetric functions can be degenerated into the big $q$-Jacobi symmetric functions, introduced in a recent paper by the second author. The latter symmetric functions admit further degenerations leading to new symmetric functions, which are analogues of $q$-Meixner and Al-Salam--Carlitz polynomials.

Each of the four families of symmetric functions ($q$-Racah, big $q$-Jacobi, $q$-Meixner, and Al-Salam--Carlitz) forms an orthogonal system of functions with respect to certain measure living on a space of infinite point configurations. The orthogonality measures of the four families are of independent interest. We show that they are linked by limit transitions which are consistent with the degenerations of the corresponding symmetric functions.
\end{abstract}

\maketitle

\tableofcontents

\section{Introduction}

\subsection{Preface}

A common feature of the \emph{hypergeometric} orthogonal polynomials (Hermite, Laguerre, Jacobi,  Charlier, Meixner, Hahn, Krawtchouk, Wilson,  Racah, and  their relatives) is that they are eigenfunctions of second order differential or difference operators and can be expressed through terminating hypergeometric series ${}_p F_q$, which range from ${}_2 F_0$ (the case of Hermite) to ${}_4 F_3$ (the case of Wilson and Racah). The hypergeometric polynomials  are conveniently organized into the \emph{Askey scheme} \cite{KS}, \cite{KLS}, which describes a multitude of limit transitions between various families. 

The Askey scheme has a $q$-version, called the  \emph{$q$-Askey scheme} \cite{KS}, \cite{KLS}. It describes a parallel, and even richer, hierarchy of the \emph{$q$-hypergeometric polynomials}. The latter ones (as suggested by their name) are eigenfunctions of  $q$-difference operators and can be expressed through terminating $q$-hypergeometric series ${}_p\phi _q$.  At the top of the $q$-Askey scheme are the Askey--Wilson and $q$-Racah polynomials, which are expressed through ${}_4\phi_3$ series. 

The goal of the present paper is to show that a partial analogue of the $q$-Askey scheme exists in the algebra of symmetric functions $\Sym$. Namely, we construct four inhomogeneous bases in $\Sym$ --- they consist of what we call  the \emph{$q$-Racah, big $q$-Jacobi, $q$-Meixner, and Al-Salam--Carlitz symmetric functions}, respectively. Each basis depends on several parameters. The basis elements are orthogonal with respect to a scalar product in $\Sym$ induced by a map $\Sym\to L^2(\Om,M)$, where $\Om$ is a space of infinite point configurations on a $q$-grid and $M$ is a probability measure on $\Om$. \footnote{For the $q$-Racah functions, this statement needs some amendment.} The bases and the orthogonality measures are linked by limit transitions
$$
\text{$q$-Racah $\longrightarrow$ big $q$-Jacobi $\longrightarrow$ $q$-Meixner $\longrightarrow$ Al-Salam--Carlitz,} 
$$
which mimic the limit transitions between the univariate $q$-hypergeometric polynomials of the same name \cite{KS}, \cite{KLS}, \cite{Koo-2011}. Informally, our hierarchy is related to the  hierarchy 
$$
{}_4\phi_3 \longrightarrow {}_3\phi_2 \longrightarrow {}_2\phi_1 \longrightarrow {}_2\phi_0 
$$
of $q$-hypergeometric functions. 

In previous work by the second author, symmetric function analogues of the big $q$-Jacobi polynomials were investigated in \cite{Ols-2017}.
The present work arose from the desire to understand whether they are an isolated phenomenon or (as we had hoped) similar orthogonal bases in $\Sym$ exist also at other levels of the $q$-hypergeometric hierarchy. We were particularly interested in symmetric functions related to Askey--Wilson/$q$-Racah polynomials, which are at the highest level ${}_4\phi_3$  of the $q$-Askey scheme (recall that the big $q$-Jacobi polynomials lie one step lower, at the ${}_3\phi_2$ level).

Let us explain our motivation a bit more. We proceeded from the analogy with the problem of harmonic analysis on infinite-dimensional classical groups and symmetric spaces \cite{Ols-2003}, \cite{BO-2005}, \cite{Cu-2018+}.  That problem asks for the decomposition of certain unitary representations into irreducible representations, and it is reduced to the description of certain random point processes --- probability measures on infinite point configurations.  The measures in question can be defined combinatorially, with the aid of discrete hypergeometric polynomials ---  of type ${}_3F_2$ in the context of \cite{BO-2005}, and of type ${}_4F_3$ in the context of \cite{Cu-2018+}.  As shown in \cite{GO-2016}, the construction of \cite{BO-2005} admits a (non-evident) $q$-version that makes use of $q$-hypergeometric polynomials of type ${}_3\phi_2$. That is why we were aimed at extending the results of \cite{GO-2016}, and its companion paper \cite{Ols-2017}, to the ${}_4\phi_3$ level. 
Then, after the discovery of the $q$-Racah orthogonal symmetric functions, we realized that one can also move in the opposite direction, going down to lower levels of the hierarchy, giving rise to a partial analogue of the $q$-Askey scheme in the ring of symmetric functions.

From the viewpoint  of the Macdonald--Koornwinder theory of multivariate basic hypergeometric polynomials, we are working with the elementary case of equal Macdonald parameters, $q=t$. However, we believe that this simplifying constraint can be removed \footnote{In what concerns the results of \cite{GO-2016}, this is  done in \cite{Ols-2018++}.}. Our main goal in the present work was to show the existence of some infinite-dimensional version of the $q$-Askey scheme, and we wanted to do it first with a minimum number of technicalities.  

Other constructions of symmetric functions out of various $N$-variate orthogonal hypergeometric polynomials were developed in \cite[Section 7]{Rains}, \cite{SV}, \cite{DH}, \cite{Ols-2011}, \cite{Ols-2012}, \cite{Petrov}. A common idea of these works (first put forward in \cite{Rains}) consists in the use of analytic continuation in the parameter $N$ (or $t^N$, in \cite{Rains}). Apart from \cite{Rains}, all the papers cited above deal with non-basic polynomials (Jacobi, Meixner, Laguerre, Charlier, and Hermite). The construction of  \cite[Section 7]{Rains} deals with Koornwinder polynomials and produces `lifted Koornwinder symmetric functions' with an additional parameter $T$, which arises from $t^N$ through analytic extrapolation. The construction of this work seems to be different from ours.

\subsection{General properties of the $q$-Racah, big $q$-Jacobi, $q$-Meixner, and Al-Salam--Carlitz symmetric functions}

We use the common notation $\Phi_\la$ for the symmetric functions in question. Here $\la$ ranges over the set $\Y$ of partitions (identified with their Young diagrams). Apart from $\la$, the functions $\Phi_\la$ depend on several continuous parameters, which we suppress from the notation. Let us describe some general properties of the functions $\Phi_\la$ and, whenever appropriate, compare them to the properties of the corresponding univariate, orthogonal polynomials.  Below, $\Sym$ denotes the algebra of symmetric functions.

\smallskip 

\noindent$\bullet$ \emph{Triangularity}. The elements $\Phi_\la$ are inhomogeneous and 
$$
\Phi_\la=S_\la+\text{lower degree terms},
$$
where $S_\la\in\Sym$ denotes the Schur symmetric functions with index $\la$ and `lower degree terms' denotes a linear combination of Schur functions $S_\mu$, with $\mu\subset\la$ (strict inclusion). 

\smallskip

\noindent$\bullet$ \emph{Expansion on interpolation symmetric functions}. The following expansion holds
\begin{equation}\label{eq1.A1}
\Phi_\la=I_\la+\sum_{\mu\subset\la}\si(\la,\mu)I_\mu,
\end{equation}
where $\si(\la,\mu)$ are certain explicit coefficients and the elements $I_\mu\in\Sym$ denote what we call the \emph{interpolation symmetric functions} --- a certain infinite-dimensional version of the multivariate symmetric interpolation polynomials studied in the 90's by  Knop, Okounkov, and Sahi (see \cite{Ok-1998}).
The interpolation symmetric functions \textit{of type A} were already present in the literature, \cite{Ols-2018+}, \cite{Cu-18}, and were used in \cite{Ols-2017} to define the big $q$-Jacobi symmetric functions.
To define the $q$-Racah symmetric functions, we had to define here another kind of interpolation symmetric functions, those of \textit{type BC}.
The expansions of the form \eqref{eq1.A1} are counterparts of the $q$-hypergeometric representations of the corresponding polynomials from the $q$-Askey scheme, i.e., the expressions in terms of terminating $q$-hypergeometric series.

\smallskip
\noindent$\bullet$ \emph{Formal orthogonality}. 
Let $\{\varphi_\ell: \ell=0,1,2,\dots\}$ be an arbitrary system of monic, univariate, orthogonal polynomials on the real line, such that $\deg \varphi_n = n$; in particular, $\varphi_0 = 1$.
A common property of these systems is that $\E(\varphi_\ell\varphi_m)=0$ for $\ell\ne m$, where $\E:\R[x]\to \R$ stands for the associated \emph{moment functional}, defined by 
$$
\text{$\E(f):=$ the coefficient of $\varphi_0=1$ in the expansion of $f\in\R[x]$ in the basis $\{\varphi_\ell\}$}.
$$
Note that $\E$ depends only on the system $\{\varphi_\ell\}$ itself and not on its orthogonality measure, which may be non-unique in the case of an indeterminate moment problem. Our systems $\{\Phi_\la:\la\in\Y\}$ enjoy a similar property: $\Phi_{\emptyset} = 1$, and $\E(\Phi_\la\Phi_\mu)=0$ for $\la\ne\mu$, with the moment functional defined by
$$
\text{$\E(F):=$ the coefficient of $\Phi_\emptyset=1$ in the expansion of $F\in\Sym$ in the basis $\{\Phi_\la\}$}.
$$
Let us emphasize that this is a special intrinsic property of the basis $\{\Phi_\la\}$. We call it \emph{formal orthogonality}. It is a necessary condition for the existence of an orthogonality measure.

\smallskip

\noindent$\bullet$
\emph{Orthogonality measures}.
In the case of univariate orthogonal polynomials on the real line, an orthogonality measure is supported by that line or by a proper subset. In our situation, when $\R[x]$ is replaced by the algebra $\Sym$, it is not a priori evident what should replace the real line. The problem is to find a probability space $(\Om,M)$ and a multiplicative linear map $F\mapsto \wt F$ from $\Sym$ to $L^2(\Om,M)$, consistent with the moment functional in the sense that
$$\E(F)=\int_{\om\in\Om} \wt F(\om)M(d\om), \quad F\in\Sym.$$
This implies that the scalar product $(F,G):=\E(FG)$ on the algebra $\Sym$, which is defined intrinsically from the system $\{\Phi_\la\}$, is induced by the scalar product of the Hilbert space $L^2(\Om,M)$, so that we may regard $M$ as an orthogonality measure. 

We show that such orthogonality measures do exist for all four families of symmetric functions $\Phi_\la$, under additional  constraints on their parameters (for formal orthogonality, we have more freedom).
In the $q$-Racah case, there is a truncation effect: the scalar product on $\Sym$, induced from $L^2(\Om,M)$, is degenerate since the image of $\Phi_\la$ in $L^2(\Om,M)$ turns out to be zero when the first coordinate of $\la$ becomes large enough. The same effect holds for the univariate $q$-Racah polynomials.

In our construction, the space $\Om$ consists of infinite point configurations on a $q$-grid $\De$.
By a point configuration, we mean a subset $X\subset\De$. The exact form of $\De$ and the support of the measure $M$ depend on the basis $\{\Phi_\la\}$.
For example, in the case of $q$-Racah and big $q$-Jacobi symmetric functions, $\De$ is the union of two geometric progressions accumulating near $0$ from both sides: 
$$
\De=\{\zeta_-, \zeta_-q,\zeta_-q^2,\dots\}\cup\{\dots, \zeta_+q^2, \zeta_+q, \zeta_+\}\subset\R_{<0}\cup\R_{>0},
$$
where $0<q<1$, and $\zeta_-<0<\zeta_+$ are (partially) determined by the parameters of the basis $\{\Phi_{\la}\}$.

In the present paper, we only establish the existence of the orthogonality measures $M$. We are planning to describe them in more detail in a later publication. Let us only point out that they are \emph{determinantal measures} whose correlation kernels can be explicitly computed in terms of basic hypergeometric functions.

\smallskip
\noindent$\bullet$ \emph{Approximation by $N$-variate symmetric polynomials}. The orthogonality properties of the symmetric functions $\Phi_\la$ are not seen from the expansion \eqref{eq1.A1}; they are established by a different construction, based on the approximation of symmetric functions $\Phi_\la$ by $N$-variate symmetric polynomials $\varphi_{\la\mid N}$. The latter polynomials are built from univariate polynomials $\varphi_\ell$ (which are taken respectively from the $q$-Racah, big $q$-Jacobi, $q$-Meixner or Al-Salam--Carlitz families) via the formula:
\begin{equation}\label{eq1.B}
\varphi_{\la\mid N}(x_1,\dots,x_N):=\frac{\det[\varphi_{\la_i+N-i}(x_j)]_{i,j=1}^N}{\prod_{1\le i<j\le N}(x_i-x_j)}.
\end{equation}
For the existence of the desired limit $\varphi_{\lambda\mid N} \to \Phi_\lambda$, it is necessary that the parameters of the univariate polynomials entering \eqref{eq1.B} vary together with $N$ in a special way; this is a subtle element of our construction. Our motivation here comes from asymptotic representation theory, where similar constructions first emerged (see \cite{BO-2005} and also \cite{GO-2016} and \cite{Cu-2018+}). It is worth noting that our construction is hinged on some fine properties of the polynomials we are dealing with.  At this point, it is not clear what exactly families of  polynomials of the $q$-Askey scheme admit such a lift to the algebra of symmetric functions.

\subsection{Structure of the paper and main theorems}

Throughout the paper, $q$ is a fixed real number with $0<q<1$. Recall that the symbol $\Y$ denotes the set of partitions (which are identified with their Young diagrams). Next, $\Y(N)\subset\Y$ is the subset of partitions with at most $N$ nonzero parts. 

In Section \ref{sect2}, we first recall some known facts about the \emph{multiparameter Schur polynomials} and related Cauchy-type identities \cite{Ok-1998}, \cite{Molev}, \cite{Ols-2018+}. With the aid of these polynomials we introduce two families of interpolation symmetric functions, denoted as $I^A_\mu(\ccdot;q)$ and $I^{BC}_\mu(\ccdot;q;s)$. Here the index $\mu$ ranges over $\Y$. 

Section \ref{sect3} is devoted to the \emph{$q$-Racah symmetric functions}. They depend, apart from $q$,  on a quadruple $(s_0,s_2,s_2,s_3)$ of complex parameters and are denoted by $\Phi^{qR}_\la(\ccdot;q; s_0,s_1,s_2,s_3)$. We define them through the expansion on $BC$-type interpolation functions. Note that $\Phi^{qR}_\la(\ccdot;q; s_0,s_1,s_2,s_3)$ is invariant under permutations of $s_i$'s. The main results of the section are Theorems \ref{thm:formalorth} and \ref{thm:qRorth}.

In Theorem \ref{thm:formalorth} we establish the formal orthogonality property of the functions $\Phi_\la^{qR}(\ccdot;q; s_0,s_1,s_2,s_3)$, $\la\in\Y$; this result holds under mild assumptions on $(s_0,s_1,s_2,s_3)$. 

In Theorem \ref{thm:qRorth}, we prove the existence of orthogonality measures under special constraints on the parameters. In particular, we require that two of them, say $s_2$ and $s_3$ (the choice is irrelevant), are real and satisfy $s_2s_3=-q^{-K}$, for some positive integer $K$. Then the index $\la$ is subject to the condition $\la_1\le K$; otherwise, $\Phi_\la^{qR}(\ccdot;q; s_0,s_1,s_2,s_3)\equiv0$ on the support of the orthogonality measure. This effect is reminiscent of the fact that the classical $q$-Racah polynomials form a finite orthogonal system, because their weight function lives on a finite set. 

In Section \ref{sect4} we deal with the \emph{big $q$-Jacobi symmetric functions} $\Phi^{bqJ}_\la(\ccdot;q;a,b,c,d)$, where $(a,b,c,d)$ is another quadruple of parameters. The functions $\Phi^{bqJ}_\la(\ccdot;q;a,b,c,d)$ are conveniently defined in terms of the expansion in the basis $\{I^A_\mu(\ccdot;q)\}$; they were introduced and studied in \cite{Ols-2017}, and we briefly overview the results of that paper.  Our new result is Theorem \ref{thm:qRbqJ}: it describes a limit transition from the $q$-Racah functions to the big $q$-Jacobi functions, which holds for their orthogonality measures as well. 

Note that in the context of classical orthogonal polynomials, the limit transition `$q$-Racah $\to$ big $q$-Jacobi' on the level of orthogonality measures appeared relatively recently (see \cite{Koo-2011}), while the earlier literature mentioned only a different kind of limit, not consistent with the measures. 

In Sections \ref{sect5} and \ref{sect6}, we show that the big $q$-Jacobi functions $\Phi^{bqJ}_\la(\ccdot;q;a,b,c,d)$ can be further degenerated as one or both of the parameters $(a,b)$ go to $0$ (Theorems \ref{thm:bqJqM} and \ref{thm:qMASC}). In this way we obtain the \emph{$q$-Meixner symmetric functions} and \emph{Al-Salam--Carlitz symmetric functions}, respectively. 

We also discuss the degeneration of the orthogonality measures of the big $q$-Jacobi functions when one or both of the parameters $(a, b)$ tend to $0$.
An interesting fact is that here (in contrast to the case of symmetric functions) the result is sensitive with respect to the tuning of the limit regime. Namely, we let the parameters $a$ and $b$ approach $0$ along some grids of the form $\const q^\Z$, and then it turns out that the limit measure depends on the choice of the grid. We announce these results in Theorems \ref{thm5.A} and \ref{thm6.B}; their proof will be given elsewhere.
(Such an effect already holds for univariate polynomials, and it is not surprising because the moment problem associated with the $q$-Meixner and Al-Salam--Carlitz polynomials is indeterminate \cite{C}.)

\subsection{Acknowledgements}

The authors are grateful to Alexei Borodin for helpful comments.
The first author was partially supported by NSF Grant DMS-1664619.

\section{Newton interpolation in $\Sym$}\label{sect2}

Fix a field $\F$ of characteristic $0$. The classical Newton interpolation polynomials with nodes $c_0,c_1,\dots\in\F$ are the polynomials
$$
(x\mid c_0, c_1,\dots)^m:=(x-c_0)(x-c_1)\dots(x-c_{m-1}), \qquad m=1,2,\dots;
$$
for $m = 0$, we set $(x\mid c_0, c_1,\dots)^0 := 1$.
The $m$th polynomial has degree $m$, highest coefficient $1$, and it vanishes at the first $m$ nodes. If the $c_i$'s are pairwise distinct, these properties characterize the polynomials uniquely.

Our aim in this section is to construct analogues of the Newton interpolation polynomials in the algebra of symmetric functions.
Let  $\Sym(N)$ be the algebra of $N$-variate symmetric polynomials over $\F$  and  $\Sym$ be the algebra of symmetric functions over $\F$.

First we briefly overview a number of known facts. 

\subsection{Multiparameter Schur polynomials}\label{sec:multischur}

(References: \cite{Ok-1998} and \cite[Section 4]{Ols-2018+}.)

Let $c_0,c_1,\dots$ be as above. The \emph{multiparameter Schur polynomials} in $N$ variables are defined by 
\begin{equation*}
s_{\mu\mid N}(x_1,\dots,x_N\mid c_0,c_1,\dots)
:= \frac{\det[(x_j\mid c_0,c_1,\dots)^{\mu_i+N-i}]}{V(x_1,\dots,x_N)}. 
\end{equation*}
Here we assume that $\mu$ is a partition of length at most $N$; overwise the polynomial is set to be equal to $0$. 
The set of partitions of length $\leq N$ will be denoted by $\Y(N)$.

The polynomials $s_{\mu\mid N}(\ccdot\mid c_0,c_1,\dots)$ are an analogue of the Newton polynomials in $\Sym(N)$.
The highest degree homogeneous component of $s_{\mu\mid N}(\ccdot\mid c_0,c_1,\dots)$ is the Schur polynomial $S_{\mu\mid N}(\ccdot)$. The following may be called the \emph{quasi-stability property}:
$$
s_{\mu\mid N}(x_1,\dots,x_N\mid c_0,c_1,\dots)\Big|_{x_N=c_0} =
s_{\mu\mid N-1}(x_1,\dots,x_{N-1}\mid c_1,c_2,\dots)
$$
(note a shift of parameters on the right-hand side).

For $\la\in\Y(N)$ set
$$
X_N(\la) := (c_{\la_1+N-1}, c_{\la_2+N-2},\dots, c_{\la_N})\in\F^N. 
$$
The following is called the \emph{extra vanishing property}: 
$$
s_{\mu\mid N}(X_N(\la)\mid c_0,c_1,\dots)=0 \quad \text{unless $\mu\subseteq\la$}
$$
(here and in the sequel we identify partitions with their Young diagrams; so $\mu\subseteq\la$ means that the diagram $\mu$ is contained in the diagram $\la$).

In particular, $s_{\mu\mid N}(X_N(\la)\mid c_0,c_1,\dots)=0$ if $|\la|<|\mu|$, where we use the standard notation
$$
|\la|:=\la_1+\la_2+\dots\,.
$$

Conversely, if the elements $c_i$ are pairwise distinct, then $s_{\mu\mid N}(\ccdot\mid c_0,c_1,\dots)$ can be characterized as the unique polynomial in $\Sym(N)$ such that its highest degree homogeneous component is the Schur polynomial $S_{\mu\mid N}(\ccdot)$ and $s_{\mu\mid N}(X_N(\la)\mid c_0,c_1,\dots)=0$, for all $\la\in\Y(N)$ with $|\la|<|\mu|$.

\subsection{Dual Schur functions and Cauchy identity} (References: \cite{Molev} and \cite[Section 4]{Ols-2018+}.)

Let $d_0,d_1,\dots$ be another sequence of parameters in $\F$. The following formula defines a family of symmetric rational
functions in $K$ variables:
\begin{equation}\label{eq.sigma}
\si_{\mu\mid K}(y_1,\dots,y_K\mid d_0, d_1, \dots)
:=\frac{\det\left[\dfrac1{(y_j\mid d_0, d_1,\dots)^{\mu_i+K-i}}\right]_{i,j=1}^K}
{\det\left[\dfrac1{(y_j\mid d_0, d_1,\dots)^{K-i}}\right]_{i,j=1}^K}.
\end{equation}
We call them the ($K$-variate) \emph{dual Schur functions} or \emph{$\si$-functions}, for short. It is assumed that $\mu\in\Y(K)$; otherwise the $\si$-function is set to be equal to $0$. If $d_0=d_1=\dots=0$, then the  $\si$-functions turn into the conventional Schur polynomials in variables $y_1^{-1},\dots,y_K^{-1}$. 

The $\si$-functions possess the following quasi-stability property
\begin{equation*}
\si_{\mu\mid K}(y_1,\dots,y_K\mid d_0, d_1,\dots)\big | _{y_K=\infty}
=\si_{\mu\mid K-1}(y_1,\dots,y_{K-1}\mid d_1,d_2,\dots).
\end{equation*}

The $\si$-functions can be regarded as elements of the algebra $\F[[y_1^{-1},\dots,y_K^{-1}]]^{S_K}$ (symmetric power series). As such, they form a \emph{topological basis} of that algebra. This means that each element of $\F[[y_1^{-1},\dots,y_K^{-1}]]^{S_K}$ can be represented, in a unique way, as a (possibly infinite) linear combination of the $\si$-functions. This follows from the fact that 
$$
\si_{\mu\mid K}(y_1,\dots,y_K\mid d_0, d_1,\dots) = S_{\mu\mid K}(y_1^{-1},\dots,y_K^{-1})+\dots,
$$
where the dots denote higher degree terms in the variables $y_i^{-1}$ (equivalently, lower degree terms in the variables $y_i$).

\begin{proposition}[Cauchy identity]\label{prop2.A}
For $N\ge K\ge1$, one has
\begin{multline}\label{eq.cauchy}
\sum_{\mu\in\Y(K)}s_{\mu \mid N}(x_1,\dots,x_N\mid c_0,c_1,\dots)
\si_{\mu \mid K}(y_1,\dots,y_K\mid c_{N-K+1},c_{N-K+2},\dots)
\\
=\prod_{j=1}^K\frac{(y_j-c_0)\dots(y_j-c_{N-1})}{(y_j-x_1)\dots(y_j-x_N)},
\end{multline}
where both sides are regarded as formal power series in $x_1,\dots,x_N,y_1^{-1},\dots,y_K^{-1}$.
\end{proposition}

Formula \eqref{eq.cauchy} may be viewed as a generating series for the multiparameter Schur polynomials. We will use it just in this role.

\subsection{Approximation $\Sym(N)\to\Sym$}\label{sec:approx}

(Reference: \cite[Subsection 2.4]{Ols-2017}.)
In this subsection we assume that $\F$ is equipped with a nontrivial topology. For instance, $\F$ is $\R$ or $\C$. Another example is $\F=\C(q)$, where $q$ is a formal parameter and the topology in $\C(q)$ is induced by the embedding into the local field $\C((q))$. Thus, for instance, the formula $1+q+q^2+\dots=\frac1{1-q}$ makes the sense in $\C(q)$.

Denote by $\Sym_{\le d}$ the subspace in $\Sym$ formed by the elements of degree at most $d$, where $d=0,1,2,\dots$\,.
In a similar way, we define the subspace $\Sym_{\le d}(N)$. These subspaces have finite dimension and the canonical projection  $\Sym\to\Sym(N)$ taking an element $F\in\Sym$ into the polynomial $F(x_1,\dots,x_N,0,0,\dots)$ determines a projection $\Sym_{\le d}\to\Sym_{\le d}(N)$. Under the condition $N\ge d$, the latter projection is a linear isomorphism and hence has the inverse:
$$
\iota_{d,N}: \Sym_{\le d}(N)\to \Sym_{\le d}, \qquad N\ge d.
$$
Evidently, $\iota_{d+1,N}$ extends $\iota_{d,N}$ for every $N\ge d+1$.

\begin{definition}\label{def2.convergence}
Let us say that a sequence $\{F_N\in\Sym(N): N\ge N_0\}$ \emph{converges} to a certain element $F\in\Sym$ if the following two conditions hold:

(1) $\sup_N\deg F_N<\infty$.

(2) For every $d$ large enough one has
$$
\lim_{N\to\infty}\iota_{d,N}(F_N) = F
$$
in the finite-dimensional $\F$-space $\Sym_{\le d}$.  Then we write $F_N\to F$ or $\lim_N F_N=F$.
\end{definition}

For further use, it is convenient to reformulate condition (2) in the following way:

\smallskip

$(2')$ For every $L=1,2,\dots$ and any vector $(x_1,\dots,x_L)\in\F^L$, the sequence $F_N(x_1,\dots,x_L,0^{N-L})$ has a limit in $\F$. (Here $0^{N-L}$ stands for the $(N-L)$-tuple $(0,\dots,0)$.)

\smallskip

Then the limit element $F\in\Sym$ is uniquely determined from the relations
$$
F(x_1,\dots,x_L,0^\infty)=\lim_{N\to\infty}F_N(x_1,\dots,x_L,0^{N-L}),
$$
where $0^\infty$ stands for the infinite sequence of $0$'s, $L=1,2,\dots$, and $(x_1,\dots,x_L)\in\F^L$.

\subsection{Large-$N$ limits of multiparameter Schur polynomials}

We keep to the assumption that $\F$ is a topological field. So far, the sequence $\{c_0,c_1,\dots\}$ was fixed, but now we suppose that it depends on a parameter $N=1,2,\dots$\,. So we write instead $\{c^{(N)}_0,c^{(N)}_1,\dots\}$.

Below $h_1,h_2,\dots$ denote the complete homogeneous symmetric functions and $H(z)$ is their generating series,
$$
H(z):=1+\sum_{k=1}^\infty h_k z^k \in\Sym[[z]].
$$

\begin{proposition}\label{prop2.B}
{\rm(i)} The following two conditions on the array $\{c^{(N)}_0,c^{(N)}_1,\dots; N=1,2,\dots\}$  guarantee the existence of the limits, in the sense of Definition \ref{def2.convergence},
$$
I_\mu:=\lim_{N\to\infty} s_{\mu\mid N}(\ccdot \mid c^{(N)}_0,c^{(N)}_1,\dots)\in\Sym
$$
for any $\mu\in\Y${\rm:}

{\rm(1)} For every $j\in\Z$, there exists a limit
$$
\wt c_j:=\lim_{N\to\infty}c^{(N)}_{N+j}.
$$

{\rm(2)} For every $k=1,2,\dots$, there exists a limit
$$
r_k:=\lim_{N\to\infty}\left\{(c^{(N)}_0)^k+\dots(c^{(N)}_{N-1})^k\right\}.
$$

{\rm(ii)} The limit elements $I_\mu\in\Sym$ satisfy the following collection of Cauchy identities{\rm:}  for every $K=1,2,\dots$ 
\begin{equation}\label{eq2.B1}
\sum_{\mu\in\Y(K)}I_\mu\si_{\mu\mid K}(y_1,\dots,y_K\mid \wt c_{-K+1}, \wt c_{-K+2},\dots)=\prod_{j=1}^K H(y_j^{-1})\exp\left\{-\sum_{k=1}^\infty \dfrac{r_k}{k y_j^k}\right\},
\end{equation}
where both sides are viewed as elements of\/ $\Sym\otimes\F[[y_1^{-1},\dots,y_K^{-1}]]^{S_K}$. 

\smallskip

{\rm(iii)} The elements $I_\mu$ are uniquely determined by these Cauchy identities. 

\smallskip

{\rm(iv)} $I_{\mu} = S_{\mu} + $lower degree terms.
\end{proposition}

\begin{proof}
(i) To prove that the elements $F_N:=s_{\mu\mid N}(\ccdot \mid c^{(N)}_0,c^{(N)}_1,\dots)$ converge it suffices to check conditions $(1)$ and $(2')$ of the previous subsection.
The first condition is evident because $F_N$ has constant degree $|\mu|$. Let us check the second condition. The elements $F_N$ are uniquely determined from the Cauchy identity \eqref{eq.cauchy} with an arbitrary fixed $K$ provided it is large enough, so let us fix such a $K$. 

Let $L$ be any positive integer and $(x_1, \ldots, x_L) \in\F^L$ be arbitrary.
Below we abbreviate $X:=(x_1,\dots,x_L)$ and use the notation
$$
H(z;X, 0^{\infty}):=1+\sum_{k=1}^\infty h_k(x_1,\dots, x_L, 0, 0, \ldots)z^k.
$$

For large $N$, the quantities 
$$
F_N(X,0^{N-L}):=s_{\mu \mid N}(X, 0^{N-L}\mid c^{(N)}_0,c^{(N)}_1,\dots)
$$ 
are uniquely determined from the relation
\begin{multline}\label{eq2.B}
\sum_{\mu\in\Y(K)}F_N(X,0^{N-L})
\si_{\mu \mid K}(y_1,\dots,y_K\mid c^{(N)}_{N-K+1},c^{(N)}_{N-K+2},\dots)
\\
=\prod_{j=1}^K H(y_j^{-1}; X,0^\infty)\cdot\prod_{j=1}^K (1-c^{(N)}_0 y_j^{-1})\dots(1-c^{(N)}_{N-1}y_j^{-1}),
\end{multline}
which follows from \eqref{eq.cauchy}.

By virtue of the definition \eqref{eq.sigma} and the assumption (1), we have 
$$
\si_{\mu \mid K}(y_1,\dots,y_K\mid c^{(N)}_{N-K+1},c^{(N)}_{N-K+2},\dots)\to\si_{\mu \mid K}(y_1,\dots,y_K\mid \wt c_{-K+1},\wt c_{-K+2},\dots)
$$
for any fixed $\mu$.

Next, by virtue of the assumption (2), we have
\begin{multline*}
\prod_{j=1}^K (1-c^{(N)}_0 y_j^{-1})\dots(1-c^{(N)}_{N-1}y_j^{-1})\\
=\prod_{j=1}^K\exp\left\{-\sum_{k=1}^\infty\frac{(c^{(N)}_0)^k+\dots(c^{(N)}_{N-1})^k}{ky_j^k}\right\}
\longrightarrow \prod_{j=1}^K\exp\left\{-\sum_{k=1}^\infty\frac{r_k}{ky_j^k}\right\},
\end{multline*}
where the convergence holds coefficient-wise in $\F[[y_1^{-1},\dots,y_K^{-1}]]$. 

These two limit relations imply that the quantities $F_N(X,0^{N-L})$ converge in $\F$, which completes the proof of claim (i).

Now claims (ii) and (iii) are evident. 

(iv) We have $s_{\mu \mid N}( \ccdot \mid c_0^{(N)}, c_1^{(N)}, \ldots) = S_{\mu\mid N}(\cdot) + $lower degree terms, for all large enough $N$, regardless of the parameters $\{c_0^{(N)}, c_1^{(N)}, \ldots\}$. Therefore the limit $I_{\mu}(\cdot)$ satisfies the analogous property.
\end{proof}

\subsection{Example: symmetric functions $I^A_\mu(\ccdot;q)$}\label{sec:IA}

We assume that one of the following two conditions holds:

$\bullet$ either $\F=\C$ and $q\in\C\setminus\{0\}$ with $|q|<1$,

$\bullet$ or $\F=\C(q)$ and the topology in $\F$ is inherited from $\C((q))$.  

Next, we take
$$
c^{(N)}_i:=q^{N-i-1}, \qquad i=0,1,2,\dots\,.
$$
The corresponding multiparameter Schur polynomials will be denoted by $I^A_{\mu\mid N}(\ccdot;q)$:
$$
I^A_{\mu\mid N}(x_1,\dots,x_N;q):=s_{\mu\mid N}(x_1,\dots,x_N\mid q^{N-1}, q^{N-2},\dots).
$$
Below, we use standard terminology for infinite $q$--Pochhammer symbols, \cite{GR}, that is:
$$(x; q)_{\infty} := \prod_{k=0}^{\infty}(1 - xq^k); \quad (x_1, \ldots, x_m; q)_\infty := \prod_{j=1}^m{(x_j; q)_\infty}.$$

\begin{proposition}\label{prop2.C}
{\rm(i)} There exist  limit  symmetric functions
\begin{equation*}
I^A_\mu(\ccdot;q):=\lim_{N\to\infty}I^A_{\mu\mid N}(\ccdot;q), \qquad \mu\in\Y.
\end{equation*}

{\rm(ii)} For every $K=1,2,\dots$, the following Cauchy identity holds
\begin{equation*}
\sum_{\mu\in\Y(K)}I^A_\mu(\ccdot;q)\si_{\mu\mid K}(y_1,\dots,y_K\mid q^{K-2}, q^{K-3},\dots)=\prod_{j=1}^K H(y_j^{-1};\ccdot)(y_j^{-1};q)_\infty.
\end{equation*}

{\rm(iii)} $I_{\mu}^A(\ccdot; q) = S_{\mu}(\ccdot)\, +\, $lower degree terms.
\end{proposition}

\begin{proof}
(i) and (iii). By virtue of Proposition \ref{prop2.B}, it suffices to check conditions (1) and (2) in item (i) of that proposition. 

In our case $c^{(N)}_{N+j}=q^{-j-1}$. These quantities do not depend on $N$, so condition (1) holds for trivial reasons, and we have $\wt c_j=q^{-j-1}$ for all $j\in\Z$. 

We proceed to condition (2). We have
$$
(c^{(N)}_0)^k+\dots+(c^{(N)}_{N-1})^k=(q^{N-1})^k+(q^{N-2})^k+\dots+1=\frac{1-q^{kN}}{1-q^k}\to \frac1{1-q^k}. 
$$
Thus, condition (2) holds with $r_k=\frac1{1-q^k}$, $k=1,2,\dots$\,. Thus, we have proved the existence of the limit symmetric functions $I^A_\mu(\ccdot;q)$.

(ii) The existence of the desired Cauchy identity follows from item (ii) of Proposition \ref{prop2.B}. The sequence $
\wt c_{-K+1}, \wt c_{-K+2},\dots$ has the desired form $q^{K-2}, q^{K-3}, \dots$\,. It remains to compute explicitly the last product on the right-hand side of \eqref{eq2.B1}. It has the form
$$
\prod_{j=1}^K\exp\left\{-\sum_{k=1}^\infty \dfrac{r_k}{k y_j^k}\right\}=\prod_{j=1}^K\exp\left\{-\sum_{k=1}^\infty \dfrac{(1-q^k)^{-1}}{k y_j^k}\right\}.
$$
It is convenient to write $(1-q^k)^{-1}$ as the infinite series $1+q^k+q^{2k}+\dots$\,.
After that, it is readily seen that the expression in question equals $\prod_{j=1}^K(y_j^{-1};q)_\infty$, as desired. 
\end{proof}

\subsection{Example: symmetric functions $I^{BC}_\mu(\ccdot;q;s)$}\label{sec:IBC}

Here we again assume that $q$ is either a nonzero complex number with $|q|<1$, or a formal parameter. We also introduce one more parameter $s$; it can be a nonzero complex number or a formal parameter.

Next, we take
$$
c^{(N)}_i:=sq^i+ s^{-1}q^{N-i-1}, \qquad i=0,1,2,\dots\,.
$$
The corresponding multiparameter Schur polynomials will be denoted by $I^{BC}_{\mu\mid N}(\ccdot;q;s)$:
$$
I^{BC}_{\mu\mid N}(x_1,\dots,x_N;q;s):=s_{\mu\mid N}(x_1,\dots,x_N\mid s+s^{-1}q^{N-1}, s q+s^{-1}q^{N-2},\dots).
$$

\begin{proposition}\label{prop2.D}
{\rm(i)} There exist  limit  symmetric functions
\begin{equation*}
I^{BC}_\mu(\ccdot;q;s):=\lim_{N\to\infty}I^{BC}_{\mu\mid N}(\ccdot;q;s), \qquad \mu\in\Y.
\end{equation*}

{\rm(ii)} For every $K=1,2,\dots$ the following Cauchy identity holds
\begin{multline}\label{eq.cauchyBC}
\sum_{\mu\in\Y(K)}I^{BC}_\mu(\ccdot;q;s)\si_{\mu\mid K}(y_1,\dots,y_K\mid s^{-1}q^{K-2}, s^{-1}q^{K-3},\dots)\\
=\prod_{j=1}^K H(y_j^{-1};\ccdot)(sy_j^{-1};q)_\infty(s^{-1}y_j^{-1};q)_\infty.
\end{multline}

{\rm{(iii)}} $I^{BC}_{\mu}(\ccdot; q; s) = S_{\mu}(\ccdot)\, + \,$lower degree terms.
\end{proposition}

\begin{proof}
We argue as in the previous proposition.

In the present context
$$
c^{(N)}_{N+j}=sq^{N+j}+s^{-1} q^{-j-1}\to s^{-1} q^{-j-1}.
$$
Thus, $\wt c_j=s^{-1} q^{-j-1}$ for $j\in\Z$. 

Next, we have
$$
(c^{(N)}_0)^k+\dots+(c^{(N)}_{N-1})^k=\sum_{i=0}^{N-1} (sq^i+ s^{-1}q^{N-i-1})^k. 
$$
Using the binomial formula, we write this as
$$
\sum_{\ell=0}^k\sum_{i=0}^{N-1} \binom{k}{\ell}(sq^i)^\ell(s^{-1}q^{N-i-1})^{k-\ell}.
$$
It is easily checked that if $0<\ell<k$, then the corresponding interior sum tends to $0$ as $N\to\infty$. The contribution from the remaining two values, $\ell=0$ and $\ell=k$, gives, in the limit, the value
$$
r_k = \frac{s^{-k}}{1-q^k} + \frac{s^k}{1-q^k}.
$$
Then the argument is completed as in the previous proposition.
\end{proof}

\subsection{Limit transition $I^{BC}_\mu(\ccdot;q;s)  \to I^A_\mu(\ccdot;q)$}
We keep to the assumptions of the previous subsection. Below, $X$ stands for the collection of arguments $(x_1,x_2,\dots)$ and $\eps\in\F$ is an additional nonzero parameter. 

\begin{proposition}\label{BCtoA}
The following limit relation holds
$$
\lim_{\eps\to 0} \eps^{|\mu|}I^{BC}_\mu(X\eps^{-1};q; \eps)=I^A_\mu(X;q).
$$
\end{proposition}

\begin{proof}
Observe that
$$
\si_{\mu\mid K}(y_1,\dots,y_K\mid s^{-1}q^{K-2}, s^{-1}q^{K-3},\dots)=s^{|\mu|}\si_{\mu\mid K}(sy_1,\dots,sy_K\mid q^{K-2}, q^{K-3},\dots).
$$
Consequently, the identity \eqref{eq.cauchyBC} can be rewritten as
\begin{multline*}
\sum_{\mu\in\Y(K)}  s^{|\mu|} I^{BC}_\mu(\ccdot;q;s)\si_{\mu\mid K}(sy_1,\dots,sy_K\mid q^{K-2}, q^{K-3},\dots)\\
=\prod_{j=1}^K H(y_j^{-1};\ccdot)(sy_j^{-1};q)_\infty(s^{-1}y_j^{-1};q)_\infty.
\end{multline*}
It follows
\begin{multline*}
\sum_{\mu\in\Y(K)}  \eps^{|\mu|} I^{BC}_\mu(X\eps^{-1};q;\eps)\si_{\mu\mid K}(\eps y_1,\dots,\eps y_K\mid q^{K-2}, q^{K-3},\dots)\\
=\prod_{j=1}^K H(y_j^{-1};X\eps^{-1})(\eps y_j^{-1};q)_\infty(\eps^{-1}y_j^{-1};q)_\infty.
\end{multline*}
Next, replace $\eps y_j$ with $y_j$ and use the obvious relation $H(z; Xa)=H(za,X)$. This gives
\begin{multline*}
\sum_{\mu\in\Y(K)}  \eps^{|\mu|} I^{BC}_\mu(X\eps^{-1};q;\eps)\si_{\mu\mid K}(y_1,\dots,y_K\mid q^{K-2}, q^{K-3},\dots)\\
=\prod_{j=1}^K H(y_j^{-1};X)(\eps^2 y_j^{-1};q)_\infty(y_j^{-1};q)_\infty.
\end{multline*}
Because
$$
\lim_{\eps\to0}(\eps^2 y_j^{-1};q)_\infty=1,
$$
we are done. 
\end{proof}

\section{$q$-Racah symmetric functions}\label{sect3}

In this section and for the remainder of this paper we assume that $q \in (0, 1)$ unless otherwise stated.

\subsection{Univariate $q$-Racah polynomials}\label{preliminariesqR}

The \emph{Askey-Wilson} $q$-difference operator \cite[Section 5]{AW-1985} is defined by 
\begin{equation}\label{eq3.A}
D^{AW} := A(u)(T_q-1)+A(u^{-1})(T_{q^{-1}}-1),
\end{equation}
where $u$ is a complex variable; $T_q$ is the $q$-shift operator acting on a test function $f$ by  $T_qf(u)=f(uq)$; finally, 
$$
A(u):=\frac{(1-t_0 u)(1-t_1 u)(1-t_2 u)(1-t_3 u)}{(1-u^2)(1-u^2 q)},
$$
where $(t_0,t_1,t_2,t_3)$ is a quadruple of complex parameters. 

The operator $D^{AW}$ acts on the space $\C[u,u^{-1}]$ of Laurent polynomials and preserves the subspace of symmetric polynomials
$$
\C[u,u^{-1}]^{sym}:=\{f\in\C[u,u^{-1}]: f(u)=f(u^{-1})\}.
$$ 

On this subspace there exists an eigenbasis $\{P_n^{AW}: n=0,1,\dots\}$,  
\begin{equation}\label{eq3.A1}
D^{AW} P^{AW}_n=-q^{-n}(1-q^n)(t_0t_1t_2t_3 q^{n-1}-1)P^{AW}_n.
\end{equation}

There is an explicit expression for the polynomials $P^{AW}_n$:
\begin{equation}\label{eq3.B}
P_n^{AW}(u)=\const {}_4\phi_3\left[\left.\begin{matrix} q^{-n}, \; t_0t_1t_2t_3q^{n-1}, \; t_0u,\; t_0u^{-1} \\ t_0t_1,\; t_0t_2,\; t_0t_3 \end{matrix}\,\right|q\right],
\end{equation}
where we use standard terminology for the basic hypergeometric functions $_4\phi_3$, as in \cite{GR}.
The choice of the constant prefactor in \eqref{eq3.B} depends on the standardization of the polynomials.
Recall, also from \cite{GR}, the notation for $q$--Pochhammer symbols that will be used hereinafter:
$$(x; q)_{n} := \prod_{k=0}^{n-1}(1 - xq^k); \quad (x_1, \ldots, x_m; q)_n := \prod_{j=1}^m{(x_j; q)_n}; \quad \textrm{for }n\in\Z_{\geq 0}.$$
A natural choice of standardization constant in \eqref{eq3.B} is
\begin{equation}\label{eq3.B1}
\const=\frac{\prod_{k=1}^3(t_0t_k; q)_n}{t_0^n (t_0t_1t_2t_3q^{n-1}; q)_n},
\end{equation}
as it gives polynomials $P_n^{AW}(u)$ with highest degree coefficient (in the variable $u+u^{-1}$) equal to $1$.

In this standardization, the polynomials $P_n^{AW}$ are symmetric with respect to all permutations of the quadruple $(t_0,t_1,t_2,t_3)$: this is evident from \eqref{eq3.A1} but not immediately visible from \eqref{eq3.B}. The link between the $q$-difference equation \eqref{eq3.A1} and the hypergeometric representation \eqref{eq3.B} is not trivial: about it, see e.g. \cite[(5.7) -- (5.9)]{AW-1985}. 

As is seen from \eqref{eq3.B} and \eqref{eq3.B1}, the coefficients of $P^{AW}_n$ depend rationally on $(t_0,t_1,t_2,t_3)$; moreover, they do not have singularities provided that $(t_0t_1t_2t_3q^{n-1};q)_n\ne0$ and $t_0\ne0$ (one can prove that the latter condition is in fact redundant and therefore it can be dropped).  

The polynomials $P^{AW}_n$ give rise both to Askey-Wilson and to $q$-Racah orthogonal  polynomials, but the corresponding  weight measures are very different. The admissible ranges of the parameters $(t_0,t_1,t_2,t_3)$ are different too.

In the case of the Askey-Wilson polynomials, one usually takes as the argument the half-sum $\frac12(u+u^{-1})$, the weight measure is supported by the interval $[-1,1]$ and has continuous density. Here, the interval $[-1,1]$ arises as the image of the unit circle $|u|=1$. 

In the case of the $q$-Racah polynomials, which we study below, the weight measure consists of finitely many atoms and $u$ should be purely imaginary. To avoid the use of purely imaginary numbers, we change the notation by setting
$$
u=iv, \quad it_0=s_0, \quad it_1=s_1, \quad it_2=s_2, \quad it_3=s_3,
$$
so that purely imaginary values of $u$ are translated into real values of $v$. Let us rewrite the polynomials $P^{AW}_n$ in the new notation and take as the argument the quantity
\begin{equation}\label{changeargument}
x:=v-v^{-1}=\frac{u+u^{-1}}i. 
\end{equation}
The resulting polynomials will be denoted by $\varphi^{qR}_n(x;q; s_0,s_1,s_2,s_3)$ (or simply $\varphi^{qR}_n$, for short). The constant prefactor is chosen from the condition that the highest degree coefficient is set to be equal to $1$.  From \eqref{eq3.A}, we obtain
\begin{multline}\label{eqn:qRpoly}
\varphi_n^{qR} \left(x; q; s_0, s_1, s_2, s_3 \right) :=
\frac{\prod_{k=1}^3(-s_0s_k; q)_n}{s_0^n (s_0s_1s_2s_3q^{n-1}; q)_n}\\
\times
{}_4\phi_3\left[\left.\begin{matrix} q^{-n}, \; s_0s_1s_2s_3q^{n-1}, \; s_0v,\; -s_0v^{-1} \\ -s_0s_1,\; -s_0s_2,\; -s_0s_3 \end{matrix}\,\right|q\right],\ n = 0, 1, \dots,
\end{multline}
with the understanding that $x$ is related to $v$ via $x=v-v^{-1}$.  

Let us emphasize again that, although $s_0$ plays a distinguished role here, the polynomials are invariant under all permutations of $(s_0,s_1,s_2,s_3)$. 

Somewhat abusing terminology, we will call the polynomials $\varphi_n^{qR} \left(x; q; s_0, s_1, s_2, s_3 \right)$ the (univariate) \emph{$q$-Racah polynomials}.  

\begin{lemma}\label{lemma3.A}
Let us interpret $s_0,s_1,s_2,s_3$ as formal parameters and work over the ground field\/ $\F:=\C(s_0,s_1,s_2,s_3)$. Next, denote by $\E=\E_{q,s_0,s_1,s_2,s_3}$ the linear functional\/ $\F[x]\to\F$ defined by
$$
\E(\varphi^{qR}_n(\ccdot;q; s_0,s_1,s_2,s_3))=\de_{n, 0}, 
$$
where $n=0,1,2,\dots$ and $\de_{n, 0}$ is the Kronecker's delta.

For arbitrary $m,n=0,1,2,\dots$, we have
$$
\E(\varphi^{qR}_m\varphi^{qR}_n)=\de_{m, n}h_n^{qR},
$$
where
\begin{equation}\label{eqn:normqR}
h^{qR}_n=h_n^{qR}(q; s_0, s_1, s_2, s_3) =
\frac{(-1)^n (q; q)_n \prod_{0 \leq i < j \leq 3}(-s_i s_j; q)_n}{(s_0s_1s_2s_3q^{n-1}; q)_n(s_0s_1s_2s_3; q)_{2n}}.
\end{equation}
\end{lemma}

\begin{proof}
As mentioned above, under suitable constraints on the parameters $(t_0,t_1,t_2,t_3)$, the polynomials $P^{AW}_n$ (viewed as functions of the variable $\frac12(u+u^{-1})$) are orthogonal with respect to a continuous weight (the Askey-Wilson weight). Consequently, they satisfy a similar orthogonality relation:
$$
\wt\E(P^{AW}_mP^{AW}_n)=\de_{m, n}\wt h_n,
$$
where $\wt\E$ is the corresponding moment functional, normalized by $\wt\E(1)=1$. From the known expression for the squared norms of the polynomials $P^{AW}_n$ (see \cite{AW-1979}, \cite[Section 3.1]{KS}, \cite[Section 14]{KLS}) we obtain
$$
\wt h_n = \frac{(q; q)_n \prod_{0 \leq i < j \leq 3}(t_i t_j; q)_n}{(t_0 t_1 t_2 t_3q^{n-1}; q)_n (t_0t_1t_2t_3; q)_{2n}}.
$$
Since the coefficients of the polynomials are rational functions of the parameters, the same formal orthogonality relations hold over the ground field $\F$. Then we can rewrite the above formula, taking into account the relationship
$$\varphi^{qR}_n(x) = \varphi^{qR}_n\left( \frac{u + u^{-1}}{i} \right) =
\left.\frac{P_n^{AW}(u)}{i^n}\right|_{t_k = -is_k, \ k=0, 1, 2, 3}$$
that follows from $(\ref{eq3.B})$, $(\ref{eq3.B1})$, $(\ref{changeargument})$ and $(\ref{eqn:qRpoly})$.
\end{proof}

\subsection{Orthogonality on a finite grid}

\begin{lemma}\label{positivecoeffs}
Let us split the quadruple $\{s_0,s_1,s_2,s_3\}$ into a disjoint union of two pairs $\{s_a, s_b\}\sqcup\{s_c, s_d\}$, $\{a, b, c, d\} = \{0, 1, 2, 3\}$. Next, suppose that inside each pair, the parameters are either real or complex-conjugate. Then the polynomials $\varphi_n^{qR}\left(x; q; s_0, s_1, s_2, s_3 \right)$ have real coefficients.
\end{lemma}
\begin{proof}
The Askey--Wilson operator $D^{AW}$ can be interpreted as an operator on $\C[x]$.  Under the above assumptions, it preserves the subspace $\R[x]$, as is readily seen from the explicit form of the coefficients $A(u)$. Assume additionally that the parameters are in general position. Then the eigenvalues of $D^{AW}$ are pairwise distinct, so that the polynomials $\varphi_n^{qR}$ can be uniquely characterized as monic polynomial eigenfunctions of $D^{AW}$. This implies that their coefficients are real. Since the coefficients depend rationally on the parameters, the additional constraint can be removed.
\end{proof}

\begin{lemma}\label{lemma3.B}
Fix  $\zeta>0$ and let $m$ range over $\Z$. Next, set 
$$
y_m:=\z q^m-\z^{-1} q^{-m}.
$$
In this notation, the correspondence 
$\Z\ni m\mapsto y_m$
is strictly decreasing, so that
\begin{equation}\label{eq3.C}
\dots<y_2<y_1<y_0<y_{-1}<y_{-2}<\dots\,.
\end{equation}
\end{lemma}

\begin{proof}
Obviously, $y_m>0$ for $m\ll 0$ and $y_m<0$ for $m\gg0$. Next, there is at most one index $m$ for which  $y_m=0$: this happens if $\z\in q^\Z:=\{q^m: m\in\Z\}$.  

Let $m$ be any index with  $y_m\ge 0$, that is,  $\z q^m\ge\z^{-1} q^{-m}$. Then 
$$
\z q^{m-1}>\z q^m\ge\z^{-1} q^{-m}> \z^{-1} q^{-m+1},
$$
whence $y_{m-1}>y_m$.

Now let $m$ be such that $y_m\le0$, that is $\z q^m\le \z^{-1} q^{-m}$. Then 
$$
\z q^{m+1}<\z q^m\le\z^{-1} q^{-m}< \z^{-1} q^{-m-1},
$$
whence $y_{m+1}<y_m$.

This implies the claim of the lemma.
\end{proof}

In the next lemma we rewrite the $q$-difference operator \eqref{eq3.A} in another form. It is convenient to multiply the operator by $-1$.

\begin{lemma}\label{lem:operatorqR}
The $q$-Racah polynomials, viewed as functions on the grid \eqref{eq3.C}, are eigenfunctions of a difference operator $D^{qR}$,
$$
D^{qR}\varphi_n^{qR}=q^{-n}(1-q^n)(s_0s_1s_2s_3 q^{n-1}-1)\varphi_n^{qR}.
$$
This operator is given by
\begin{equation}\label{eq3.DqR}
D^{qR}f(y_m)=\be_m [f(y_{m+1})-f(y_m)]+ \de_m [f(y_{m-1})-f(y_m)],
\end{equation}
where
\begin{gather*}
\be_m:= -\frac{(1-s_0 \z q^m)(1-s_1 \z q^m)(1-s_2 \z q^m)(1-s_3 \z q^m)}{(1+\z^2 q^{2m})(1+\z^2 q^{2m+1})},\\ 
\de_m:= -\frac{(1+s_0 \z^{-1} q^{-m})(1+s_1 \z^{-1} q^{-m})(1+s_2 \z^{-1} q^{-m})(1+s_3 \z^{-1} q^{-m})}{(1+\z^{-2} q^{-2m})(1+\z^{-2} q^{-2m+1})}.
\end{gather*}
\end{lemma}

\begin{proof}
Follows directly from \eqref{eq3.A}.
\end{proof}

We are going to impose sufficient constraints on $(s_0,s_1,s_2,s_3)$ under which the first $(K+1)$ polynomials $\varphi_0^{qR}, \ldots, \varphi_K^{qR}$ become orthogonal on a given $(K+1)$-point interval of the grid $\{y_m: m\in\Z\}$, $K=0,1,2,\dots$\,.
Any such interval is determined by two integers $L\ge R$ with $L-R=K$ and has the form
\begin{equation}\label{eq3.D}
\widetilde{\De}_{L,R}:=\{y_L, y_{L-1},\dots, y_{R+1}, y_R\}
\end{equation}
(thus, the indices $L$ and $R$ correspond to the left and right ends of the interval).

\begin{definition}\label{positivity}
We say that a quadruple $(s_0, s_1, s_2, s_3)$ is \textit{admissible} if $s_2 = -\zeta q^R$, $s_3 = \zeta^{-1} q^{-L}$, for some integers $L \geq R$, and the pair $(s_0, s_1)$ satisfies one of the following conditions:
\begin{itemize}
	\item $s_0 = \overline{s_1} \in \C \setminus \R$;
	\item $\zeta^{-1}q^{m+1} < s_0, s_1 < \zeta^{-1}q^{m}$, for some $m\in\Z$;
	\item $-\zeta q^{m} < s_0, s_1 < -\zeta q^{m+1}$, for some $m\in\Z$.
\end{itemize}
We also set $K:=L-R$.
\end{definition}

For admissible quadruples of parameters the following claims hold true:

1. The $q$-Racah polynomials are well defined. Indeed,  their coefficients do not have singularities --- this follows from  \eqref{eqn:qRpoly} and the fact that $s_0s_1s_2s_3<0$.

2. As a consequence, the conclusion of Lemma \ref{lemma3.A} holds if $(s_0,s_1,s_2,s_3)$ is an admissible quadruple of numerical parameters, as opposed to formal variables.

3. The coefficients of the $q$-Racah polynomials are real. Indeed, this is an immediate corollary of Lemma \ref{positivecoeffs}.

4.  Because $s_0s_1s_2s_3<0$, the eigenvalues $q^{-n}(1-q^n)(s_0s_1s_2s_3q^{n-1}-1)$, $n=0,1,2,\dots$, strictly decrease and hence are pairwise distinct. 

\begin{proposition}\label{prop:positivity}
Fix an arbitrary finite interval $\widetilde{\De}_{L,R}$ of the grid $\{y_m: m\in\Z\}$ as defined above.
Let $(s_0, s_1, s_2, s_3)$ be admissible with $s_2 = -\zeta q^R$, $s_3 = \zeta^{-1}q^{-L}$.

Then there exists a unique probability measure $w^{qR} = w^{qR}(\ccdot; q; \zeta; s_0, s_1, s_2, s_3)$ on the grid $\{y_m: m\in\Z\}$, supported by $\widetilde{\De}_{L,R}$, and such that for any polynomial $f$,  
$$
\E(f)=\sum_{m\in\Z}f(y_m)w^{qR}(y_m; q; \zeta; s_0, s_1, s_2, s_3)=\sum_{y\in\wt\De_{L,R}}f(y)w^{qR}(y; q; \zeta; s_0, s_1, s_2, s_3),
$$
where $\E$ is the formal moment functional defined in Lemma \ref{lemma3.A}. 
\end{proposition}

\begin{proof}
Let us prove the existence claim, the uniqueness being obvious.
Recall that $\beta_m, \delta_m$, $m\in\Z$, are the coefficients of the difference operator $D^{qR}$; see Lemma $\ref{lem:operatorqR}$.
With $s_2=-\zeta q^R$ and $s_3=\zeta^{-1} q^{-L}$, we have $\de_R=0$ and  $\be_L=0$.
We claim that $\be_R, \be_{R+1}, \ldots, \be_{L-1}$ and $\de_{R+1}, \ldots, \de_{L-1}, \de_L$ are all strictly positive.
As the proof is similar for both sets of $K$ numbers, we focus on the $\be$'s and leave the $\de$'s to the reader.
The denominator of $\beta_m$ is always strictly positive and so is the factor $(1 - s_2\zeta q^m)$.
Since $1 - s_3\zeta q^m = 1 - q^{m-L}$ and $0 < q < 1$, the factor $(1 - s_3\zeta q^m)$ is strictly negative for $m = R, R+1, \ldots, L-1$.
Finally, if $s_0, s_1$ satisfy any of the three conditions in Definition $\ref{positivity}$, then $(1 - s_0\zeta q^m)(1 - s_1\zeta q^m)$ is always strictly positive.
Because of the minus sign in front,  it follows that $\beta_m > 0$ for $m = R, R+1, \ldots, L-1$.

There exist positive numbers $w_{R}, w_{R+1}, \ldots, w_{L-1}, w_{L}$ satisfying the conditions:

\smallskip

$\bullet$ $w_m\be_m=w_{m+1}\de_{m+1}$, $R\le m\le L-1$ (the balance relation);

$\bullet$ $\sum_{m=R}^L w_m=1$ (normalization).

\smallskip

Indeed, this holds with
$$
w_m=\const\prod_{i=R}^{m-1}\be_i\cdot\prod_{j=R+1}^m \de_j^{-1}, \qquad R\le m\le L,
$$
and an appropriate normalization constant.
Next,  set $w_m=0$ for all $m\in\Z\setminus[R,L]$ and observe that with this definition, the balance relation holds for any $m\in\Z$, because $\de_R=\be_L=0$. From this and the very definition of $D^{qR}$, see \eqref{eq3.DqR}, it follows that for arbitrary functions $f$  and $g$ on the grid $\{y_m: m\in\Z\}$, one has the symmetry relation
$$
\sum_{m\in\Z}(D^{qR}f)(y_m)g(y_m)w_m=\sum_{m\in\Z}f(y_m)(D^{qR}g)(y_m)w_m.
$$
Substituting $f=\varphi^{qR}_k$, $g=\varphi^{qR}_n$ with $k\ne n$, and using the fact that the eigenvalues of $D^{qR}$ are pairwise distinct, it follows that the polynomials $\varphi^{qR}_k$ and $\varphi^{qR}_n$ are orthogonal with respect to the weight function given by the $w_m$'s.
Then the probability measure with the weights
$$
w^{qR}(y_m;q;\zeta;s_0,s_1,s_2,s_3):=w_m, \qquad m\in\Z, 
$$
has the desired properties.
\end{proof}

\begin{remark}
One can derive an explicit formula for $w^{qR}(y_m; q; \zeta; s_0, s_1, s_2, s_3)$:
\begin{gather*}
w^{qR}(y_m; q; \zeta; s_0, s_1, s_2, s_3) =
\frac{\prod_{k = 0}^3(-s_k^{-1}\zeta q, s_k^{-1}\zeta^{-1}q; q)_{\infty}}
{\prod_{0 \leq i < j \leq 3}(-s_i^{-1}s_j^{-1}q; q)_{\infty}} \cdot
\frac{(q/(s_0s_1s_2s_3); q)_{\infty}}{(q, -\zeta^2 q, -\zeta^{-2}q; q)_{\infty}}\\
\times
\frac{(s_0\zeta, s_1\zeta, s_2\zeta, s_3\zeta; q)_m}
{(-q\zeta/s_0, -q\zeta/s_1, -q\zeta/s_2, -q\zeta/s_3; q)_m}
\frac{1+\zeta^2q^{2m}}{1+\zeta^2}
\left( \frac{q}{s_0s_1s_2s_3} \right)^m.
\end{gather*}
Note that the right-hand side vanishes unless $R\le m\le L$: this is due to the factor $(s_3\zeta;q)_m$ in the numerator and the factor $(-\zeta/s_2;q)_m$ in the denominator.
The fact that the weights above add up to $1$ is exactly Bayley's formula for a very well-poised $_6\psi_6$ hypergeometric series, see \cite{Ba} and \cite[Section 5]{GR}.
\end{remark}

\begin{corollary}\label{cor3.A}
We keep to the assumptions and notation of Proposition \ref{prop:positivity}. Recall that $K=L-R$ and note that the grid $\wt\De_{L,R}$ consists of $K+1$ points. Let $\ell^2(\wt\De_{L,R}, w^{qR})$ denote the $(K+1)$-dimensional real space of functions on $\wt\De_{L,R}$ with the inner product
$$
(f,g):=\sum_{y\in\wt\De_{L,R}}f(y)g(y)w^{qR}(y;q;\zeta;s_0,s_1,s_2,s_3).
$$

{\rm(i)} The first $K+1$ polynomials $\varphi^{qR}_n$, $0\le n\le K$, form an orthogonal basis in $\ell^2(\wt\De_{L,R}, w^{qR})$.

{\rm(ii)} Their squared norms are the quantities $h_n^{qR}$, $0\le n\le K$,  defined in \eqref{eqn:normqR}.

{\rm(iii)} The polynomials $\varphi^{qR}_n$ with indices $n>K$ vanish identically on $\wt\De_{L,R}$. 
\end{corollary}

\begin{proof}
Proposition \ref{prop:positivity} tells us that on the space of polynomials, the inner product $(f,g)$ coincides with $\E(fg)$. This allows us to apply Lemma \ref{lemma3.A} and conclude that 
$(\varphi^{qR}_m, \varphi^{qR}_n)=\de_{m, n}h^{qR}_n$ for all $m,n=0,1,2,\dots$\,.
One can verify from $(\ref{eqn:normqR})$ that $h^{qR}_n > 0$ for all $n = 0, 1, \ldots, K$.
Together with the orthogonality relations, this shows that the polynomials $\varphi^{qR}_n$, $n = 0, 1, \ldots, K$, are linearly independent as functions on the grid $\wt\De_{L, R}$.
Since $\wt\De_{L,R}$ consists of $K+1$ points, it follows that the first $K+1$ polynomials $\varphi^{qR}_n$, restricted to $\wt\De_{L,R}$, form a basis in $\ell^2(\wt\De_{L,R}, w^{qR})$. This proves (i) and (ii).

Next, recall that $s_2 = -\zeta q^R$, $s_3 = \zeta^{-1} q^{-L}$. From  \eqref{eqn:normqR} it is seen that $h^{qR}_n=0$ for all $n>K$, because of the factor 
$$
(-s_2s_3;q)_n=(q^{R-L};q)_n=(q^{-K};q)_n.
$$ 
This proves (iii). (Alternatively, (iii) can be deduced from \eqref{eqn:qRpoly}.)
\end{proof}

\begin{remark}
Based on the material of this subsection, e.g. Corollary $\ref{cor3.A}$, it would make sense to call the polynomials $\varphi^{qR}_n(x; q; s_0, s_1, s_2, s_3)$ the $q$-Racah polynomials only when $s_2 = -\zeta q^{R}$ and $s_3 = \zeta^{-1}q^{-L}$, for some real $\zeta > 0$ and integers $L \geq R$, and to call them the Askey-Wilson polynomials in the generic case.
For instance, for different conditions on the parameters $s_0, s_1, s_2, s_3$, a similar convention is followed in \cite{vDS}.
For simplicity, we use the name $q$-Racah polynomials, regardless of the choice of parameters $s_0, s_1, s_2, s_3$.
\end{remark}

\subsection{Multivariate $q$-Racah polynomials}\label{sec:multivariateqR}

Given $N \geq 1$ and $\la\in\Y(N)$, we can define the \textit{multivariate $q$-Racah polynomial} $\varphi^{qR}_{\la\mid N}(x_1, \ldots, x_N; q; s_0, s_1, s_2, s_3)$ in accordance with \eqref{eq1.B}.
We will be interested in a special case when the parameters $s_0, s_1, s_2, s_3$ are dependent on $N$ as follows:
\begin{multline}\label{eqn:multiqR}
\varphi^{qR}_{\la\mid N}(x_1, \ldots, x_N; q; s_0q^{\frac{1-N}{2}}, s_1q^{\frac{1-N}{2}}, s_2q^{\frac{1-N}{2}}, s_3q^{\frac{1-N}{2}}) :=\\
\frac{\det\left[ \varphi^{qR}_{\ell_j}(x_i; q;s_0q^{\frac{1-N}{2}},s_1q^{\frac{1-N}{2}},s_2q^{\frac{1-N}{2}},s_3q^{\frac{1-N}{2}}) \right]_{i, j = 1}^{N}}{V(x_1, \ldots, x_N)},
\end{multline}
where $\la\in \Y(N)$, $\ell_j = \la_j + N - j$ for $j = 1, \ldots, N$, and $V(x_1, \ldots, x_N)$ is the Vandermonde determinant.
The reason for the scaling factor $q^{\frac{1-N}{2}}$ that affects the four $q$-Racah parameters will be clear later (see Theorem $\ref{thm:qRpolyslimit}$ below).
The multivariate $q$-Racah polynomials $(\ref{eqn:multiqR})$ are well-defined whenever $s_0 \neq 0$ and $s_0s_1s_2s_3 \notin q^{\Z}$, as it can be seen from $(\ref{eqn:qRpoly})$.

In order to talk about orthogonality measures for the multivariate $q$-Racah polynomials, we need to make some restrictions.
We assume that $N$ is an odd integer.
Assume also that $(s_0, s_1, s_2, s_3)$ is admissible in the sense of Definition $\ref{positivity}$, in particular, $s_2 = -\zeta q^{R}, s_3 = \zeta^{-1} q^{-L}$, for some integers $L \geq R$.
We are restricting $N$ to be an odd integer so that
$$
s_2q^{\frac{1-N}{2}} = -\zeta q^{R - \frac{N-1}{2}},\quad s_3q^{\frac{1-N}{2}} = \zeta^{-1} q^{-(L+\frac{N-1}{2})}
$$
have the same form as the parameters $s_2, s_3$, that is, 
$$
s_2q^{\frac{1-N}{2}}\in -\zeta q^{\Z}, \quad  s_3q^{\frac{1-N}{2}} \in \zeta^{-1}q^{\Z}.
$$
We know that the univariate polynomials on the right-hand side of \eqref{eqn:multiqR} are well defined and have real coefficients; hence the same holds for the $N$-variate polynomials  $\varphi^{qR}_{\la \mid N}$, too.

For any odd integer $N \geq 1$, let $\widetilde{\Omega}_N^{qR}$ be the set of $N$-point configurations on the grid
$$
\widetilde{\De}_{L+\frac{N-1}{2}, R - \frac{N-1}{2}} :=\left \{ y_{L + \frac{N-1}{2}}, \; y_{L + \frac{N-3}{2}}, \; \ldots, \; y_{R - \frac{N-3}{2}},\; y_{R - \frac{N-1}{2}} \right\},
$$
and let $M_N^{qR} = M_N^{qR}(\ccdot; q; \zeta; s_0, s_1, s_2, s_3)$ be the probability measure on $\widetilde{\Omega}^{qR}_N$ given by
$$
M_N^{qR}(X; q; \zeta; s_0, s_1, s_2, s_3) := \const\cdot V(X)^2 \prod_{x\in X}{ w^{qR}(x; q; \zeta; s_0q^{\frac{1-N}{2}}, \dots, s_3q^{\frac{1-N}{2}}) },
$$
for any $X\in\widetilde{\Omega}^{qR}_N$. The constant above is chosen to make $M_N^{qR}$ a probability measure.

Let $L^2(\wt\Omega^{qR}_{N}, M^{qR}_N)$ denote the finite-dimensional real Hilbert space of functions on $\wt\Om^{qR}_N$, with the inner product defined by the weight measure $M^{qR}_N$. As before, we set $K:=L-R$. Let $\Y_K(N)$ denote the set of Young diagrams contained in the $N\times K$ rectangle, that is, 
$$
\Y_K(N):=\{\la\in\Y(N): \la_1\le K\}.
$$

The next proposition is an extension of Corollary \ref{cor3.A}. 

\begin{proposition}\label{prop:multivariateqR}
We regard the multivariate $q$-Racah polynomials defined in $(\ref{eqn:multiqR})$ as functions on\/ $\widetilde\Omega^{qR}_N$. 

{\rm(i)} The polynomials with indices $\la\in\Y_K(N)$ form an orthogonal basis in $L^2(\wt\Om^{qR}_N, M^{qR}_N)$.

{\rm(ii)} Their squared norms are the quantities 
$$
h^{qR}_{\la\mid N}(q; s_0, s_1, s_2, s_3):=
\prod_{j = 1}^N{ \frac{h^{qR}_{\la_j + N - j}(q; s_0q^{\frac{1-N}{2}}, s_1q^{\frac{1-N}{2}}, s_2q^{\frac{1-N}{2}}, s_3q^{\frac{1-N}{2}})}
{h^{qR}_{N - j}(q; s_0q^{\frac{1-N}{2}}, s_1q^{\frac{1-N}{2}}, s_2q^{\frac{1-N}{2}}, s_3q^{\frac{1-N}{2}})} }.
$$

{\rm(iii)} The polynomials with indices in $\Y(N)\setminus\Y_K(N)$ vanish identically on $\wt\Om^{qR}_{N}$.
\end{proposition}

\begin{proof}
These are direct consequences of the corresponding claims of Corollary \ref{cor3.A}. 

We begin with item (iii). Set $L':=L+\frac{N-1}2$, $R':=R-\frac{N-1}2$.  If $\la\in\Y(N)\setminus\Y_K(N)$, then $\la_1+N-1>L'-R'$. By claim (iii) of Corollary \ref{cor3.A}, the polynomial $\varphi^{qR}_{\la_1+N-1}$ vanishes on $\wt\De_{L',R'}$.  It follows that the determinant on the right-hand side of \eqref{eqn:multiqR} vanishes provided that all arguments $x_j$ are in $\wt\De_{L',R'}$, which proves (iii). 

Observe that by restricting all $N$-variate symmetric polynomials to $\wt\Om^{qR}_N$ we obtain all functions on this finite set. Together with item (iii) just proved, this implies that the polynomials $\varphi^{qR}_{\la\mid N}$ with indices $\la\in\Y_K(N)$ span the whole space of functions on $\wt\Om^{qR}_N$. Since $|\Y_K(N)|=|\wt\Om^{qR}_N|$, we see that these polynomials form a basis of that space. 

Now the proof of (i) and (ii) is achieved by the following standard argument.  Let us abbreviate
\begin{gather*}
\varphi_n=\varphi^{qR}_n, \quad \varphi_{\la\mid N}=\varphi^{qR}_{\la\mid N}, \quad \De=\widetilde{\De}_{L',R'}, \quad \Om_N=\wt\Om^{qR}_N, \\
w(x)=w^{qR}(x; q; \zeta; s_0q^{\frac{1-N}{2}}, \dots, s_3q^{\frac{1-N}{2}}), \quad M_N(X)=M_N^{qR}(X; q; \zeta; s_0, s_1, s_2, s_3),\\
h_n=h^{qR}_{n}(q; s_0q^{\frac{1-N}{2}}, s_1q^{\frac{1-N}{2}}, s_2q^{\frac{1-N}{2}}, s_3q^{\frac{1-N}{2}}).
\end{gather*}
Given $\la,\mu\in\Y(N)$, the inner product $(\varphi_{\la\mid N}, \varphi_{\mu\mid N})$ in $L^2(\Om_N,M_N)$ is:
\begin{gather*}
(\varphi_{\la\mid N}, \varphi_{\mu\mid N})=\sum_{X\in\Om_N}\varphi_{\la\mid N}(X)\varphi_{\mu\mid N}(X) M_N(X)\\
=\const \sum_{\substack{x_1<\dots<x_N\\ x_1,\dots,x_N\in\De}}\det[\varphi_{\la_j+N-j}(x_i)]\det[\varphi_{\mu_j+N-j}(x_i)]w(x_1)\dots w(x_N)\\
=\frac{\const}{N!} \sum_{x_1,\dots,x_N\in\De}\det[\varphi_{\la_j+N-j}(x_i)]\det[\varphi_{\mu_j+N-j}(x_i)]w(x_1)\dots w(x_N).
\end{gather*}
Expanding the determinants and using the orthogonality relations for the univariate polynomials gives
$\const\, \de_{\la, \mu}\prod_{j=1}^N h_{\la_j+N-j}$. Next, when $\la$ and $\mu$ are the empty partition, the result should be equal to $1$, whence $\const=1/\prod_{j=1}^N h_{N-j}$. Therefore, we get 
$$
(\varphi_{\la\mid N}, \varphi_{\mu\mid N})=\de_{\la, \mu}\prod_{j=1}^N\dfrac{h_{\la_j+N-j}}{h_{N-j}},
$$
which completes the proof of (i) and (ii). 
\end{proof}

\begin{remark}
Similar results to Proposition $\ref{prop:multivariateqR}$ are proved in \cite{vDS}, for $q$-Racah polynomials with the additional parameter $t$.
\end{remark}

The squared norms $h_n(q; s_0q^{\frac{1-N}{2}}, s_1q^{\frac{1-N}{2}}, s_2q^{\frac{1-N}{2}}, s_3q^{\frac{1-N}{2}})$ have closed form formulas, as shown in Corollary \ref{cor3.A}.
To write a closed formula for $h^{qR}_{\la\mid N}(q; s_0, s_1, s_2, s_3)$, the following definition is convenient.

\begin{definition}\label{def:pochlambda}
For any partition $\mu$, the rational function $(x; q)_{\mu}$ is defined by
$$(x; q)_{\mu} := \prod_{i = 1}^{\ell(\mu)}{(xq^{1-i}; q)_{\mu_i}}.$$
Moreover, we use the standard notation: $(x_1, \ldots, x_m; q)_{\mu} := \prod_{i=1}^m{(x_i; q)_{\mu}}$.
Also define $$\widehat{\mu} := (2\mu_1, 2\mu_1, 2\mu_2, 2\mu_2, 2\mu_3, \dots)$$ as the partition whose Young diagram is the Young diagram of $\mu$ after replacing each square by a $2\times 2$ square.
\end{definition}

Then one can verify, for any partition $\la\in\Y(N)$,
\begin{equation}\label{eqn:hlambdaqR}
h^{qR}_{\la\mid N}(q; s_0, s_1, s_2, s_3) = (-1)^{|\la|} \cdot 
\frac{(s_0s_1s_2s_3 q^{-N}, q^{N}; q)_{\la}\prod_{0 \leq i < j \leq 3}(-s_i s_j; q)_{\la}}
{(s_0s_1s_2s_3; q)_{\widehat{\la}}}.
\end{equation}

\subsection{Construction of the $q$-Racah symmetric functions}\label{sec:construction}

For each $\la\in\Y$, we construct an element $\Phi^{qR}_{\la}(\ccdot; q; s_0, s_1, s_2, s_3)\in\Sym$ that arises as a limit of renormalized, multivariate $q$-Racah polynomials.
In this subsection, $s_0, s_1, s_2, s_3$ can be any complex numbers such that $s_0 \neq 0$ and $s_0s_1s_2s_3 \notin q^{\Z}$.

\begin{definition}\label{df:qRfunction}
For any partition $\la$, we define the \textit{$q$-Racah symmetric function} $\Phi^{qR}_{\la}(\ccdot; q; s_0, s_1, s_2, s_3) \in \Sym$ by
\begin{equation*}
\Phi^{qR}_{\la}(\ccdot; q; s_0, s_1, s_2, s_3) := \sum_{\mu\subseteq\la}
\sigma^{qR}(\la, \mu; q; s_0, s_1, s_2, s_3) I_{\mu}^{BC}(\ccdot; q; s_0),
\end{equation*}
where $I_{\mu}^{BC}(\ccdot; q; s_0) \in\Sym$ are defined in Section $\ref{sec:IBC}$, and
\begin{multline}\label{eqn:sigmaqR}
\sigma^{qR}(\la, \mu; q; s_0, s_1, s_2, s_3) := \\
s_0^{|\mu| - |\la|} \det\left[ \frac{(-s_0 s_1q^{\mu_k+1-k}, -s_0 s_2q^{\mu_k+1-k}, -s_0 s_3q^{\mu_k+1-k}; q)_{\la_j - j - \mu_k + k}}{q^{(\mu_k + 1 - k)(\la_j - j - \mu_k + k)}(q, s_0 s_1 s_2 s_3 q^{\la_j + \mu_k + 1 - j - k}; q)_{\la_j - j - \mu_k + k}} \right]_{j, k = 1}^{\ell(\la)}
\end{multline}
for partitions $\mu \subseteq \la$, with the understanding that $\sigma^{qR}(\emptyset, \emptyset; q; s_0, s_1, s_2, s_3) := 1$.
\end{definition}

Note that the expression on the right-hand of \eqref{eqn:sigmaqR} automatically vanishes if $\mu\not\subseteq\la$ and that $\sigma^{qR}(\la, \la; q; s_0, s_1, s_2, s_3) = 1$: this follows from Lemma \ref{lem:matrices} below. Note also that
$$
\Phi^{qR}_{\la}(\ccdot; q; s_0, s_1, s_2, s_3) = S_{\la} + \textrm{lower degree terms}, \textrm{ for all }\la\in\Y,
$$
because of  Proposition \ref{prop2.D} (iii).  It follows that $\{ \Phi^{qR}_{\la}(\ccdot; q; s_0, s_1, s_2, s_3) \}_{\la\in\Y}$ is an inhomogeneous basis of $\Sym$.

\begin{lemma}\label{lem:matrices}
Let $f_{j, k}(\ell, m)$, $1 \leq j,k \leq N$, be complex-valued functions that take pairs $(\ell, m)\in (\Z_{\geq 0})^2$ as arguments and such that $f_{j, k}(0, 0) = 1$.
For any two partitions $\la, \mu\in\Y(N)$, define
\begin{equation}\label{sigma:munu}
\rho(\la, \mu) := \det\left[ \frac{f_{j, k}(\la_j, \mu_k)}{(q; q)_{\la_j - \mu_k + k - j}} \right]_{j, k = 1}^{N}.
\end{equation}
Then $\rho(\la, \mu) = 0$, unless $\mu \subseteq \la$.
Moreover, if $\mu \subseteq \la$, then
\begin{equation}\label{sigma:munu2}
\rho(\la, \mu) = \det\left[ \frac{f_{j, k}(\la_j, \mu_k)}{(q; q)_{\la_j - \mu_k + k - j}} \right]_{j, k = 1}^{\ell(\la)}.
\end{equation}
Finally,  $\rho(\la,\la)=1$. 
\end{lemma}

\begin{proof}
Denote by $A=[A_{jk}]$ the $N\times N$ matrix on the right-hand side of \eqref{sigma:munu}.

Suppose $\mu\not\subseteq\la$. This means that there exists $i\in\{1,\dots,N\}$ such that $\la_i<\mu_i$. Then, for any pair $(j,k)$ with $j\ge i\ge k$, we have  $\la_j<\mu_k$ and hence $\la_j-\mu_k+k-j<0$, which entails $A_{jk}=0$,  because $(q;q)_p^{-1}=0$ for $p\in\Z_{<0}$. It follows that $\det A=0$, which proves the first claim.

Suppose now $\mu\subseteq\la$, in particular, $\ell(\mu)\le\ell(\la)$. A similar argument (together with the assumption $f_{j, k}(0, 0) = 1$) shows that, for any pair $(j,k)$ such that $j>\ell(\la)$ and $j\ge k$, one has $A_{jk}=\de_{jk}$. This proves the second claim.

Finally, if $\la=\mu$, then $A$ is strictly upper unitriangular, so that $\det A=1$. This proves the third claim. 
\end{proof}

\begin{theorem}\label{thm:qRpolyslimit}
For any partition $\la\in\Y$ and $N \geq \ell(\la)$, denote the renormalized, multivariate $q$-Racah polynomial
$$q^{\frac{N-1}{2}|\la|}\cdot\varphi_{\la | N}^{qR}
(x_1q^{\frac{1-N}{2}}, \dots, x_Nq^{\frac{1-N}{2}}; q; s_0q^{\frac{1-N}{2}}, s_1q^{\frac{1-N}{2}}, s_2 q^{\frac{1-N}{2}}, s_3q^{\frac{1-N}{2}}).$$
by $\widetilde\varphi_{\la | N}^{qR}(\ccdot; q; s_0, s_1, s_2, s_3) = \widetilde\varphi_{\la | N}^{qR}(x_1, \dots, x_N; q; s_0, s_1, s_2, s_3)$.
Then
$$
\widetilde\varphi_{\la | N}^{qR}(\ccdot; q; s_0, s_1, s_2, s_3) \rightarrow \Phi^{qR}_{\la}(\ccdot; q; s_0, s_1, s_2, s_3)
$$
in the sense of Definition $\ref{def2.convergence}$.
\end{theorem}
\begin{proof}
\textit{Step 1.}
The set $\{I_{\mu\mid N}^{BC}(\ccdot; q; s_0)\}_{\mu\in\Y(N)}$ of polynomials in Section $\ref{sec:IBC}$ is a basis of $\Sym(N)$.
Then, for any $\la\in\Y(N)$, there is a unique expansion of the form
\begin{multline}\label{eqn:expansionphi}
\widetilde{\varphi}^{qR}_{\la\mid N}(x_1, \ldots, x_N; q; s_0, s_1, s_2, s_3) =\\
\sum_{\mu\in\Y(N)}{ \sigma^{qR}_N(\la, \mu; q; s_0, s_1, s_2, s_3) I_{\mu\mid N}^{BC}(x_1, \ldots, x_N; q; s_0) },
\end{multline}
for some coefficients $\sigma^{qR}_N(\la, \mu; q; s_0, s_1, s_2, s_3)$.
We will prove that
\begin{equation}\label{eqn:vanishingsigma}
\sigma^{qR}_N(\la, \mu; q; s_0, s_1, s_2, s_3) = 0, \textrm{ unless }\mu\subseteq\la,
\end{equation}
and
\begin{equation}\label{eqn:equalitysigma}
\lim_{N \rightarrow \infty}{\sigma^{qR}_N(\la, \mu; q; s_0, s_1, s_2, s_3)} =
\sigma^{qR}(\lambda, \mu; q; s_0, s_1, s_2, s_3), \textrm{ for any }\mu\subseteq\la.
\end{equation}
Proposition $\ref{prop2.D}$ proves the limits
\begin{equation}\label{BClimits}
I^{BC}_{\mu\mid N}(x_1, \ldots, x_N; q; s_0) \rightarrow I^{BC}_{\mu}(x_1, x_2, \ldots; q; s_0),\ \mu\in\Y,
\end{equation}
in the sense of Definition $\ref{def2.convergence}$.

From $(\ref{eqn:vanishingsigma})$, the number of summands in the right-hand side of $(\ref{eqn:expansionphi})$ is finite and independent of $N$.
Then, because of $(\ref{eqn:equalitysigma})$ and $(\ref{BClimits})$, the theorem follows.
All that remains is to prove $(\ref{eqn:vanishingsigma})$ and $(\ref{eqn:equalitysigma})$; in the next steps, we will prove both of these statements.

\smallskip

\textit{Step 2.}
From $(\ref{eqn:qRpoly})$, we have
\begin{equation}\label{varphiexpansion}
\varphi_{\ell}^{qR}(x; q; s_0, s_1, s_2, s_3) = \sum_{m=0}^{\ell}{ c(\ell, m)
(x \mid s_0 + s_0^{-1}, s_0q + s_0^{-1}q^{-1}, \dots )^m },
\end{equation}
where, by denoting $\chi := s_0s_1s_2s_3$,
\begin{equation*}
c(\ell, m) := \frac{(q; q)_{\ell}}{(q; q)_m} \cdot \frac{\prod_{i=1}^3(-s_0s_iq^m; q)_{\ell - m}}
{q^{m(\ell - m)}s_0^{\ell - m}(q, \chi q^{\ell + m - 1}; q)_{\ell - m}}, \ \ell \geq m \geq 0.
\end{equation*}
Do the change of variables $x\mapsto xq^{\frac{1-N}{2}}$ in the expansion $(\ref{varphiexpansion})$, and use
\begin{multline*}
(xq^{\frac{1-N}{2}} \mid s + s^{-1}, sq + s^{-1}q^{-1}, \ldots)^m\\
= q^{\frac{(1-N)m}{2}} (x \mid sq^{\frac{N-1}{2}} + s^{-1}q^{\frac{N-1}{2}}, sq^{\frac{N+1}{2}} + s^{-1}q^{\frac{N-3}{2}}, \ldots)^m,
\end{multline*}
to obtain
\begin{multline}\label{varphiexpansion2}
q^{\frac{N-1}{2}\ell}\cdot\varphi_{\ell}^{qR}(xq^{\frac{1-N}{2}}; q; s_0, s_1, s_2, s_3) =\\
\sum_{m=0}^{\ell}{ q^{\frac{N-1}{2}(\ell - m)} c(\ell, m)
(x \mid s_0q^{\frac{N-1}{2}} + s_0^{-1}q^{\frac{N-1}{2}}, s_0q^{\frac{N+1}{2}} + s_0^{-1}q^{\frac{N-3}{2}}, \ldots)^m }.
\end{multline}
Further, do the change $s_k \mapsto s_kq^{\frac{1-N}{2}}$, for $k = 0, 1, 2, 3$, in $(\ref{varphiexpansion2})$ to obtain
\begin{multline}\label{varphiexpansion3}
q^{\frac{N-1}{2}\ell}\cdot\varphi_{\ell}^{qR}(xq^{\frac{1-N}{2}}; q; s_0q^{\frac{1-N}{2}}, s_1q^{\frac{1-N}{2}}, s_2q^{\frac{1-N}{2}}, s_3q^{\frac{1-N}{2}}) =\\
\sum_{m=0}^{\ell}{ q^{\frac{N-1}{2}(\ell - m)} c_N(\ell, m)
(x \mid s_0 + s_0^{-1}q^{N-1}, s_0q + s_0^{-1}q^{N-2}, \ldots)^m },
\end{multline}
where
\begin{gather}
c_N(\ell, m) := \left. c(\ell, m) \right|_{s_k \mapsto s_kq^{\frac{1-N}{2}}, \; \forall  k = 0, 1, 2, 3}\nonumber\\
= \frac{(q; q)_{\ell}}{(q; q)_{m}} \cdot
\frac{\prod_{i=1}^3(-s_0s_i q^{m+1-N}; q)_{\ell - m}}{q^{m(\ell - m)}(s_0q^{\frac{1-N}{2}})^{\ell - m}(q, \chi q^{\ell+m+1-2N}; q)_{\ell - m}}\nonumber\\
= q^{\frac{N-1}{2}(\ell - m)} \frac{(q; q)_{\ell}}{(q; q)_{m}} \cdot
\frac{s_0^{m - \ell}\prod_{i=1}^3(-s_0s_i q^{m+1-N}; q)_{\ell - m}}{q^{m(\ell - m)}(q, \chi q^{\ell+m+1-2N}; q)_{\ell - m}}.\label{cN.ellm}
\end{gather}

\smallskip

\textit{Step 3.}
We generalize $(\ref{varphiexpansion3})$ to several variables by making use of the following general fact.
Let $\{g_m^{(i)}\}_{m \geq 0; N \geq i \geq 1}$ be an array of variables, and $\{f_{j}^{(i)}\}_{N \geq i, j \geq 1}$ be the linear combinations
\begin{equation}\label{CB1}
f^{(i)}_j := \sum_{m=0}^{\infty}{a^{(j)}_m g_m^{(i)}},
\end{equation}
for some complex numbers $a^{(j)}_m$, $j = 1, \ldots, N$, $m\geq 0$, such that $a^{(j)}_m = 0$ for large enough $m$ (depending on $j$).
Then
$$\det[ f_j^{(i)} ]_{i, j = 1}^N = \sum_{0 \leq m_N < \ldots < m_1}{\det[ a^{(j)}_{m_k} ]_{j, k = 1}^N \det[ g^{(i)}_{m_k} ]_{i, k = 1}^N}.$$
This identity is the well-known \textit{Cauchy-Binet formula}. Note that the last sum above is finite, given the vanishing condition on the coefficients $a_m^{(j)}$.

Let $x_1, \ldots, x_N$ be $N$ variables; let the array $\{g^{(i)}_{m}\}_{m \geq 0; N \geq i \geq 1}$ consist of the following  polynomials
$$g_m^{(i)} := (x_i \mid s_0 + s_0^{-1}q^{N-1}, s_0q + s_0^{-1}q^{N-2}, \ldots)^m.$$
Moreover, the array of linear combinations $\{f_j^{(i)}\}_{N \geq i, j \geq 1}$ consists of the polynomials
$$f_j^{(i)} := q^{\frac{N-1}{2}\ell_j} \cdot \varphi_{\ell_j}^{qR}(x_i q^{\frac{1-N}{2}}; q; s_0q^{\frac{1-N}{2}}, s_1q^{\frac{1-N}{2}}, s_2q^{\frac{1-N}{2}}, s_3q^{\frac{1-N}{2}}),$$
where $\la$ is a partition of length at most $N$, and $\ell_j := \la_j + N - j$, $j = 1, \ldots, N$.
Next, because of $(\ref{varphiexpansion3})$, the coefficients $\{a^{(j)}_m\}_{m \geq 0; N \geq j \geq 1}$ that satisfy $(\ref{CB1}) $ are
$$a_m^{(j)} := q^{\frac{N-1}{2}(\ell_j - m)}c_N(\ell_j, m) \mathbf{1}_{\{ m \leq \ell_j \}}.$$
Note that the formula for $c_N(\ell_j, m)$ in $(\ref{cN.ellm})$ contains the factor $(q; q)_{\ell_j - m}^{-1}$, which vanishes whenever $m > \ell_i$.
Therefore, the factor $\mathbf{1}_{\{m \leq \ell_i\}}$ is irrelevant in the last display and can be removed, i.e.,
$$a_m^{(j)} = q^{\frac{N-1}{2}(\ell_j - m)}c_N(\ell_j, m) =
q^{(N-1)(\ell_j - m)} \frac{(q; q)_{\ell_j}}{(q; q)_{m}} \cdot
\frac{s_0^{m - \ell_j}\prod_{i=1}^3(-s_0s_i q^{m+1-N}; q)_{\ell_j - m}}{q^{m(\ell_j- m)}(q, \chi q^{\ell_j + m + 1 - 2N}; q)_{\ell_j - m}}.$$
Then the Cauchy-Binet formula gives
\begin{multline*}
\det\left[ q^{\frac{N-1}{2}\ell_j} \cdot \varphi_{\ell_j}^{qR}(x_i q^{\frac{1-N}{2}}; q; s_0q^{\frac{1-N}{2}}, s_1q^{\frac{1-N}{2}}, s_2q^{\frac{1-N}{2}}, s_3q^{\frac{1-N}{2}}) \right]_{i, j = 1}^N =\\
\sum_{0 \leq m_N < \ldots < m_1}{ \det\left[ a^{(j)}_{m_k} \right]_{j, k = 1}^N
\det\left[ (x_i \mid s_0 + s_0^{-1}q^{N-1}, s_0q + s_0^{-1}q^{N-2}, \ldots)^{m_k} \right]_{i, k = 1}^N }.
\end{multline*}
There is a bijection between $N$-tuples of integers $m_1 > \ldots > m_N \geq 0$ and partitions $\mu\in\Y(N)$ via $m_k := \mu_k + N - k$,  $k = 1, 2, \ldots, N$.
Thus after dividing the previous identity by the Vandermonde determinant $V(x_1, \ldots, x_N)$, we obtain
\begin{multline}\label{varphiexpansion4}
\widetilde{\varphi}^{qR}_{\la\mid N}(x_1, \ldots, x_N; q; s_0, s_1, s_2, s_3)\\
= q^{\frac{N-1}{2}|\la|} \cdot \frac{\det\left[ \varphi_{\ell_j}^{qR}(x_i q^{\frac{1-N}{2}}; q; s_0q^{\frac{1-N}{2}}, s_1q^{\frac{1-N}{2}}, s_2q^{\frac{1-N}{2}}, s_3q^{\frac{1-N}{2}}) \right]_{i, j = 1}^N}{V(x_1q^{\frac{1-N}{2}}, \ldots, x_N q^{\frac{1-N}{2}})}\\
= \frac{\det\left[ q^{\frac{N-1}{2}\ell_j} \cdot \varphi_{\ell_j}^{qR}(x_i q^{\frac{1-N}{2}}; q; s_0q^{\frac{1-N}{2}}, s_1q^{\frac{1-N}{2}}, s_2q^{\frac{1-N}{2}}, s_3q^{\frac{1-N}{2}}) \right]_{i, j = 1}^N}{V(x_1, \ldots, x_N)}\\
=\sum_{\mu\in\Y(N)}{ \sigma_N^{qR}(\la, \mu; q; s_0, s_1, s_2, s_3)
I_{\mu\mid N}^{BC}(x_1, \ldots, x_N; q; s_0) },
\end{multline}
where
\begin{align}
&\sigma^{qR}_N(\la, \mu; q; s_0, s_1, s_2, s_3) = \det\left[ a^{(j)}_{m_k} \right]_{j, k = 1}^N\nonumber\\
&= \det\left[ q^{(N-1)(\ell_j - m_k)} \frac{(q; q)_{\ell_j}}{(q; q)_{m_k}} \cdot
\frac{s_0^{m_k - \ell_j}\prod_{i=1}^3(-s_0s_i q^{m_k+1-N}; q)_{\ell_j - m_k}}{q^{m_k(\ell_j - m_k)}(q, \chi q^{\ell_j + m_k + 1 - 2N}; q)_{\ell_j - m_k}} \right]_{j, k=1}^{N}\nonumber\\
&= \det\left[ \frac{(q; q)_{\ell_j}}{(q; q)_{m_k}} \cdot
\frac{s_0^{m_k - \ell_j}\prod_{i=1}^3(-s_0s_i q^{m_k+1-N}; q)_{\ell_j - m_k}}{q^{(m_k+1-N)(\ell_j - m_k)}(q, \chi q^{\ell_j + m_k + 1 - 2N}; q)_{\ell_j - m_k}} \right]_{j, k=1}^{N}\nonumber\\
&= \prod_{i=1}^N{\frac{(q; q)_{\la_i + N - i}}{(q; q)_{\mu_i + N - i}}} s_0^{|\mu| - |\la|}\nonumber\\
&\ \ \ \times \det\left[ \frac{\prod_{i=1}^3(-s_0s_i q^{\mu_k+1-k}; q)_{\la_j - \mu_k + k - j}}{q^{(\mu_k + 1 - k)(\la_j - \mu_k + k - j)}(q, \chi q^{\la_j + \mu_k + 1 - j - k}; q)_{\la_j - \mu_k + k - j}} \right]_{j, k=1}^{N}.\label{detNN}
\end{align}

\smallskip

\textit{Step 4.}
From the last display above and Lemma $\ref{lem:matrices}$, applied to 
$$f_{j, k}(\ell, m) := \frac{\prod_{i=1}^3(-s_0s_i q^{m+1-k}; q)_{\ell - m + k - j}}{q^{(m + 1 - k)(\ell - m + k - j)}(q, \chi q^{\ell + m + 1 - j - k}; q)_{\ell - m + k - j}},$$
it follows that $\sigma^{qR}_N(\la, \mu; q; s_0, s_1, s_2, s_3) = 0$, unless $\mu\subseteq\la$; this shows the desired $(\ref{eqn:equalitysigma})$.
Moreover the determinant of the $N\times N$ matrix $(\ref{detNN})$ is equal to the determinant of its top-left $\ell(\la)\times\ell(\la)$ submatrix. Therefore
\begin{equation}\label{sigmaNsigma}
\sigma^{qR}_N(\la, \mu; q; s_0, s_1, s_2, s_3)
= \prod_{i=1}^N{\frac{(q; q)_{\la_i+N-i}}{(q; q)_{\mu_i+N-i}}} \cdot \sigma^{qR}(\la, \mu; q; s_0, s_1, s_2, s_3).
\end{equation}
The product in the last display has $N$ terms, but for $i > \ell(\lambda)$ each of the terms are equal to $1$.
It follows that we can take the limit of that product:
$$\lim_{N \rightarrow \infty}{ \prod_{i=1}^N{ \frac{(q; q)_{\lambda_i + N - i}}{(q; q)_{\mu_i + N - i}} } }
= \lim_{N \rightarrow \infty}{ \prod_{i=1}^{\ell(\lambda)}{ \frac{(q; q)_{\lambda_i + N - i}}{(q; q)_{\mu_i + N - i}} } }
= \prod_{i=1}^{\ell(\la)} { \frac{(q; q)_{\infty}}{(q; q)_{\infty}} } = 1.
$$
This proves the desired limit $\sigma^{qR}_N \rightarrow \sigma^{qR}$ in $(\ref{eqn:equalitysigma})$ and finishes the proof.
\end{proof}

\subsection{Formal orthogonality}\label{sec:formalorth}

In this subsection we are working over the ground field $\F:=\C$ and keep the same assumptions on the complex parameters $s_0, s_1, s_2, s_3$ as in the previous subsection, that is,  $s_0\ne0$ and $s_0s_1s_2s_3\notin q^\Z$.  Let $\Phi^{qR}_{\la}$ be the shorthand notation for the $q$-Racah symmetric function $\Phi^{qR}_{\la}(\ccdot; q; s_0, s_1, s_2, s_3)$ (Definition \ref{df:qRfunction}).
We know that $\{ \Phi^{qR}_{\la} \}_{\la\in\Y}$ is a basis of $\Sym$, so we can define the associated \textit{moment functional} $\E : \Sym \rightarrow \C$ by
$$\E(\Phi^{qR}_{\la}) := \delta_{\la, \emptyset}.$$
Attached to this functional is the inner product
\begin{equation}\label{eqn:inner}
(F, G):= \E(FG),\quad F,G\in\Sym.
\end{equation}

Recall that the symbol$(\ccdot; q)_{\la}$ and the partition $\widehat{\la}$ are exhibited in Definition \ref{def:pochlambda}. Below we also use the standard notation \cite{Mac}
$$
n(\la):=\sum_{i=1}^{\ell(\la)}(i-1)\la_i = \sum_{i=1}^{\ell(\la')}\begin{pmatrix} \la'_i \\ 2\end{pmatrix}.
$$

\begin{theorem}\label{thm:formalorth}
For $\la,\mu\in\Y$, we have
\begin{equation}\label{eqn:formalorth}
\left(\Phi^{qR}_{\la}, \Phi^{qR}_{\mu} \right) = \delta_{\la, \mu}h_{\la}^{qR}(q; s_0, s_1, s_2, s_3),
\end{equation}
where
\begin{equation}\label{eqn:hqR}
h_{\la}^{qR}(q; s_0, s_1, s_2, s_3) := q^{n(\la') - n(\la)} \cdot \frac{\prod_{0 \leq i < j \leq 3}(-s_i s_j; q)_{\la}}{(s_0s_1s_2s_3; q)_{\widehat{\la}}} \left( \frac{s_0s_1s_2s_3}{q} \right)^{|\la|}.
\end{equation}
\end{theorem}

\begin{proof}
The key idea is to approximate the functional $\E : \Sym \rightarrow \C$ above by functionals $\E_N : \Sym(N) \rightarrow \C$, as $N$ tends to infinity.

\smallskip

\textit{Step 1.}
Define the $N$-variate polynomial $\widetilde\varphi_{\la | N}^{qR}(\ccdot; q; s_0, s_1, s_2, s_3)\in \Sym(N)$ by
\begin{multline}\label{phiNRacah}
\widetilde\varphi_{\la | N}^{qR}(x_1, \ldots, x_N; q; s_0, s_1, s_2, s_3) :=\\
q^{\frac{N-1}{2}|\la|}\cdot\varphi_{\la | N}^{qR}
(x_1q^{\frac{1-N}{2}}, \dots, x_Nq^{\frac{1-N}{2}}; q; s_0q^{\frac{1-N}{2}}, s_1q^{\frac{1-N}{2}}, s_2 q^{\frac{1-N}{2}}, s_3q^{\frac{1-N}{2}}),
\end{multline}
for any $\la\in\Y(N)$.
We can abbreviate $\widetilde\varphi_{\la | N}^{qR}(\ccdot; q; s_0, s_1, s_2, s_3)$ by $\widetilde\varphi^{qR}_{\la | N}$.
The set $\{ \widetilde\varphi_{\la | N}^{qR} \}_{\la\in\Y(N)}$ is a basis of $\Sym(N)$.
Then we can define $\E_N : \Sym \rightarrow \C$ by
$$\E_N(\widetilde\varphi_{\la | N}^{qR}) = \delta_{\la, \emptyset}, \textrm{ for all } \lambda\in\Y(N).$$
Therefore, for any $\la\in\Y$, we have
\begin{equation}\label{eqn:limitEqR}
\E(\Phi^{qR}_{\la}) = \lim_{N \rightarrow \infty} \E_N(\widetilde\varphi_{\la | N}^{qR}),
\end{equation}
since both sides are equal to $\delta_{\la, \emptyset}$.
Moreover, Theorem $\ref{thm:qRpolyslimit}$ shows, for any $\la\in\Y$,
\begin{equation}\label{eqn:limitqR}
\Phi^{qR}_{\la} = \lim_{N \rightarrow \infty}{\widetilde\varphi_{\la | N}^{qR}}.
\end{equation}

\textit{Step 2.}
Let $\la, \mu$ be any two partitions; also let $N$ range over integers larger than $\max\{\ell(\la), \ell(\mu)\}$.
The generic argument \cite[proof of Theorem 4.1, (c)]{Ols-2017} shows that $(\ref{eqn:limitqR})$ implies that the coefficients in the expansions of $\widetilde\varphi_{\la | N}^{qR} \widetilde\varphi_{\mu | N}^{qR}$ in the basis $\{ \widetilde\varphi_{\nu | N}^{qR} \}_{\nu\in\Y(N)}$ converge to the corresponding coefficients in the expansion of $\Phi^{qR}_{\la}\Phi^{qR}_{\mu}$ in the basis $\{ \Phi^{qR}_{\nu} \}_{\nu\in\Y}$.
Consequently, $(\ref{eqn:limitEqR})$ can be generalized to
\begin{equation}\label{eqn:keylimit}
\E(\Phi^{qR}_{\la}\Phi^{qR}_{\mu}) = \lim_{N \rightarrow \infty}
{\E_N(\widetilde\varphi_{\la | N}^{qR}\widetilde\varphi_{\mu | N}^{qR})},
\end{equation}
for any partitions $\la, \mu$.

\smallskip

\textit{Step 3.}
In the case when $(s_0, s_1, s_2, s_3)$ is an admissible quadruple in the sense of Definition \ref{positivity}, it follows from Proposition \ref{prop:multivariateqR} and the definition \eqref{phiNRacah} that
\begin{equation}\label{eqn:keyequality}
\E_N(\widetilde\varphi_{\la | N}^{qR} \widetilde\varphi_{\mu | N}^{qR}) =
\delta_{\la, \mu}\cdot q^{\frac{N-1}{2}(|\la| + |\mu|)} h^{qR}_{\la\mid N}(q; s_0, s_1, s_2, s_3).
\end{equation}
Since the both sides are rational functions $s_0, s_1, s_2, s_3$, the equality holds even when $(s_0, s_1, s_2 ,s_3)$ is not admissible, but a generic quadruple of complex numbers.

Now $(\ref{eqn:keylimit})$ and $(\ref{eqn:keyequality})$ already show the desired orthogonality relation $(\ref{eqn:formalorth})$ when $\la \neq \mu$.
It remains to prove $(\ref{eqn:formalorth})$ when $\la = \mu$.

\smallskip

\textit{Step 4.}
From $(\ref{eqn:keylimit})$ and $(\ref{eqn:keyequality})$, we have
\begin{equation}\label{step4}
\left(\Phi^{qR}_{\la}, \Phi^{qR}_{\la} \right) = \mathbb{E}( (\Phi_{\la}^{qR})^2 ) = \lim_{N \rightarrow \infty}{\E_N((\widetilde\varphi_{\la\mid N}^{qR})^2)}
= \lim_{N \rightarrow \infty}{q^{(N-1)|\la|} h^{qR}_{\la\mid N}(q; s_0, s_1, s_2, s_3) }.
\end{equation}
The explicit expression for $h^{qR}_{\la\mid N}(q; s_0, s_1, s_2, s_3)$ in $(\ref{eqn:hlambdaqR})$ reads
$$
h^{qR}_{\la\mid N}(q; s_0, s_1, s_2, s_3) = (-1)^{|\la|} \cdot 
\frac{(\chi q^{-N}, q^{N}; q)_{\la}\prod_{0 \leq i < j \leq 3}(-s_i s_j; q)_{\la}}
{(\chi; q)_{\widehat{\la}}},
$$
where we denoted $\chi := s_0s_1s_2s_3$.
We need to consider the asymptotics of this formula as $N$ tends to infinity.
Clearly, $\lim_{N \rightarrow \infty}(q^N; q)_{\lambda} = 1$.
The only other term depending on $N$ is
\begin{equation}\label{chipoch}
(\chi q^{-N}; q)_{\lambda} = \prod_{i = 1}^{\ell(\lambda)}{(\chi q^{-N+1-i}; q)_{\lambda_i}}.
\end{equation}
The $i$-th term in the product above is
\begin{equation}\label{eqn:ithterm}
\begin{aligned}
(\chi q^{-N+1-i}; q)_{\la_i}
&= (1 - \chi q^{-N+1-i})\cdots (1 - \chi q^{-N+\la_i-i})\\
&= (-1)^{\la_i}  (\chi q^{-N+1-i})^{\la_i} q^{{\la_i \choose 2}} (1 + O(q^N))\\
&= (-1)^{\la_i}  q^{(1-N)\la_i} \left( \frac{\chi}{q} \right)^{\la_i} q^{-(i-1)\la_i} q^{{\la_i \choose 2}} (1 + O(q^N)).
\end{aligned}
\end{equation}
Since $\sum_{i=1}^N{\la_i} = |\la|$, $\sum_{i=1}^N{(i-1)\la_i} = n(\la)$ and $\sum_{i=1}^N{ {\la_i \choose 2} } = n(\la')$, the asymptotic equality $(\ref{eqn:ithterm})$ shows
$$
(\ref{chipoch}) = (-1)^{|\lambda|}q^{(1-N)|\lambda|}(\chi/q)^{|\lambda|}q^{-n(\lambda)+n(\lambda')}(1+o(1)).
$$
It follows that the limit in the right-hand side of $(\ref{step4})$ exists and is given by the right-hand side of \eqref{eqn:hqR}.
\end{proof}

\begin{corollary}
Suppose that the parameters satisfy the additional constraints $s_0s_1s_2s_3 \neq 0$ and $-s_is_j\notin q^{\Z}$, for all pairs $1\le i<j\le3$. Then the inner product \eqref{eqn:inner} on $\Sym$ is nondegenerate, and $\{\Phi^{qR}_\la: \la\in\Y\}$ is an orthogonal basis with respect to it.
\end{corollary}
\begin{proof} Indeed, we only need to check that these constraints guarantee that $(\Phi^{qR}_{\la}, \Phi^{qR}_{\la})\ne0$ for all $\la\in\Y$, but this is immediate from \eqref{eqn:hqR}.
\end{proof}

\subsection{Orthogonality measures for admissible parameters $(s_0,s_1,s_2,s_3)$}
Given two integers $L\ge R$, we consider the doubly infinite grid
$$
\De_{L, R} := \{ -\zeta^{-1}q^{-L}, -\zeta^{-1}q^{1-L}, -\zeta^{-1}q^{2-L}, \ldots \} \sqcup \{ \ldots, \zeta q^{R+2}, \zeta q^{R+1}, \zeta q^R \}
$$
on the real line. Let $\Omega^{qR}$ denote the set of all point configurations on $\Delta_{L, R}$.
We equip $\Omega^{qR}$ with the topology coming from its natural identification with $\{ 0, 1 \}^{\De_{L, R}}$, so that $\Omega^{qR}$ is a compact, metrizable space.

\begin{theorem}\label{thm:qRorth}
Let $(s_0, s_1, s_2, s_3)$ be admissible in the sense of Definition \ref{positivity}; in particular, $s_2 = -\zeta q^{R}$, $s_3 = \zeta^{-1}q^{-L}$, for some $\zeta>0$ and integers $L \geq R$.

There exists a unique probability measure $M^{qR}_{q, \zeta, s_0, s_1, s_2, s_3}$ on $\Omega^{qR}$ such that
\begin{equation}\label{eqn:emptydelta0}
\langle M^{qR}_{q, \zeta, s_0, s_1, s_2, s_3} , \Phi^{qR}_{\lambda}  \rangle
= \delta_{\lambda, \emptyset}, \quad \la\in\Y,
\end{equation}
where we denoted $\Phi_{\lambda}^{qR} := \Phi_{\lambda}^{qR}(\ccdot; q; s_0, s_1, s_2, s_3)$.
\end{theorem}

Before proceeding to the proof, let us explain the simple idea behind it.
Below, the parameter $N$ ranges over the odd integers $1, 3, 5, \ldots$.
Recall the notation $y_m:=\zeta q^m - \zeta^{-1} q^{-m}$ introduced in Lemma \ref{lemma3.A}. In this notation, the points of the finite grid $\wt\De_{L + \frac{N-1}{2}, R - \frac{N-1}{2}}$ are
$$
y_{L + \frac{N-1}2} < y_{L + \frac{N-3}2} <\dots <y_{R-\frac{N-3}2}<y_{R -\frac{N-1}2}.
$$
Let us set
$$
y_m^{(N)}:=y_m q^{\frac{N-1}2}, \quad m\in\Z.
$$
In this notation, the points of the rescaled grid 
$q^{\frac{N-1}{2}}\widetilde{\De}_{L + \frac{N-1}{2}, R - \frac{N-1}{2}}$
are
$$
y^{(N)}_{L + \frac{N-1}2} < y^{(N)}_{L + \frac{N-3}2} <\dots <y^{(N)}_{R-\frac{N-3}2}<y^{(N)}_{R -\frac{N-1}2}.
$$

For each odd integer $N=1,3,5,\dots$, we define an order-preserving embedding $f_N$ of  the rescaled grid $q^{\frac{N-1}2}\widetilde{\De}_{L + \frac{N-1}{2}, R - \frac{N-1}{2}}$  into the infinite grid $\De_{L, R}$. As $N$ goes to infinity, the image of the grid $q^{\frac{N-1}2}\widetilde{\De}_{L + \frac{N-1}{2}, R - \frac{N-1}{2}}$ grows and in the limit covers the whole infinite grid $\De_{L,R}$. 

Using these embeddings, we can put all the probability measures $M_N^{qR}(\ccdot; q; \zeta; s_0, s_1, s_2, s_3)$ into the common space $\Omega^{qR}$.  Then, with the aid of Theorem \ref{thm:qRpolyslimit}, we show that the resulting measures weakly converge, as $N\to\infty$, to  the (unique) measure on $\Om^{qR}$ satisfying \eqref{eqn:emptydelta0}. 
\smallskip

Next, we need a preparation. The embedding $f_N : q^{\frac{N-1}2}\widetilde{\De}_{L+\frac{N-1}{2}, R-\frac{N-1}{2}} \rightarrow \De_{L, R}$ mentioned above is defined by
\begin{equation}\label{eqn:fN}
f_N( y_m^{(N)} ) = \left \{
  \begin{aligned}
    &-\zeta^{-1} q^{-m + \frac{N-1}{2}} && \text{if } m \geq 0;\\
    &\zeta q^{m + \frac{N-1}{2}} && \text{if } m < 0.
\end{aligned} \right.
\end{equation}
Note that for any fixed $n \geq 1$ and for $N$ large enough, $f_N$ takes the $n$th leftmost point of $q^{\frac{N-1}2}\widetilde{\De}_{L+\frac{N-1}{2}, R-\frac{N-1}{2}}$ to the $n$th leftmost point of $\De_{L,R}$, and likewise for rightmost points. 

The following lemma shows that, as $N$ gets large, the finite grid $q^{\frac{N-1}2}\widetilde{\De}_{L+\frac{N-1}{2}, R-\frac{N-1}{2}}$ becomes very close to the infinite grid $\De_{L,R}$.

\begin{lemma}\label{lemma3.C1}
The following bound holds uniformly on $y\in q^{\frac{N-1}2}\widetilde{\De}_{L+\frac{N-1}{2}, R-\frac{N-1}{2}}${\rm:}
$$
y=f_N(y)+O(q^{N/2}).
$$
\end{lemma}

\begin{proof}
We have $y=y^{(N)}_m$, i.e.,
$$
y = y^{(N)}_m=\zeta q^{m+\frac{N-1}2}-\zeta^{-1} q^{-m+\frac{N-1}2},
$$
for some $m\in\Z$. Comparing with \eqref{eqn:fN}, we see that 
\begin{equation*}
y^{(N)}_m-f_N( y_m^{(N)} ) = \begin{cases} \zeta q^{m + \frac{N-1}{2}}, & m\ge0\\
-\zeta^{-1} q^{-m + \frac{N-1}{2}}, & m<0.
\end{cases}
\end{equation*}
Obviously, in both cases the difference is  $O(q^{N/2})$. 
\end{proof}

Note that near the ends, a stronger estimate holds (the difference is of order $O(q^N)$), but we do not use this fact. The next lemma is a corollary of the previous one.

\begin{lemma}\label{lemma3.C2}
Let $p_k\in\Sym$ denote the $k$th power sum, $k=1,2,\dots$\,. For fixed $k$ and an arbitrary $N$-point configuration $X_N$ in $q^{\frac{N-1}2}\widetilde{\De}_{L+\frac{N-1}{2}, R-\frac{N-1}{2}}$, the following bound holds as $N\to\infty${\rm:}
$$
p_k(X_N)=p_k(f_N(X_N))+o(1),
$$ 
uniformly on $X_N$. 
\end{lemma}

\begin{proof}
Note that all the grids $q^{\frac{N-1}{2}}\widetilde{\De}_{L+\frac{N-1}{2}, R-\frac{N-1}{2}}$ and $\De_{L,R}$ are contained in some compact interval $[-C,C]\subset\R$. From this fact and Lemma \ref{lemma3.C1}, we obtain that for $y\in q^{\frac{N-1}2}\widetilde{\De}_{L+\frac{N-1}{2}, R-\frac{N-1}{2}}$, one has
$$
y^k=(f_N(y))^k+O(q^{N/2}), \quad N\to\infty,
$$
uniformly on $y$. Because $N q^{N/2}\to0$ we are done. 
\end{proof}

We need the following refinement of the previous lemma.

\begin{lemma}\label{lemma3.C3}
Let\/ $\Phi_N\in\Sym(N)$, $N=1,3,5,\dots$,  be a sequence of elements which converge to an element\/ $\Phi\in\Sym$ in the sense of Definition \ref{def2.convergence}. For an arbitrary $N$-point configuration $X_N$ in $q^{\frac{N-1}2}\widetilde{\De}_{L+\frac{N-1}{2}, R-\frac{N-1}{2}}$, the following bound holds as $N\to\infty${\rm:}
$$
\Phi_N(X_N)=\Phi(f_N(X_N))+o(1),
$$ 
uniformly on $X_N$. 
\end{lemma}

We have used the following convention in the statement of the lemma above.
If $F\in\Sym$ and $X$ is an $N$-point configuration, then $F(X)$ is the evaluation of $F$ in the infinite sequence $X, 0, 0, \ldots$.

\begin{proof}
From Definition \ref{def2.convergence}, it is seen that we may assume, without loss of generality, that $\Phi$ is a monomial in power sums, $\Phi=p_{k_1}\dots p_{k_n}$, and $\Phi_N$ has exactly the same form (where the power sums are restricted to $N$ variables). Next, observe that for $k$ fixed, the quantities $|p_k(X)|$, where $X$ is an arbitrary configuration in $\De_{L,R}$, are uniformly bounded. Together with Lemma \ref{lemma3.C2}, this gives the desired result. 
\end{proof}

\begin{proof}[Proof of Theorem $\ref{thm:qRorth}$]
Let us abbreviate $M_N:= M_N^{qR}(\ccdot; q; \zeta; s_0, s_1, s_2, s_3)$. The natural bijection 
$$
\wt\De_{L + \frac{N-1}{2}, R - \frac{N-1}{2}} \longrightarrow q^{\frac{N-1}2}\wt\De_{L + \frac{N-1}{2}, R - \frac{N-1}{2}}
$$
allows us to carry over $M_N$ to the space of $N$-point configurations in $q^{\frac{N-1}2}\wt\De_{L + \frac{N-1}{2}, R - \frac{N-1}{2}}$. Let $M'_N$ denote the resulting probability measure and $f_N(M'_N)$ be its pushforward under $f_N$ --- it is a probability measure on the space $\Om^{qR}$. 

Let $\{\Phi_N\}$ and $\Phi$ be as in Lemma \ref{lemma3.C3}. It follows from this lemma that
$$
\langle M'_N, \Phi_N\rangle=\langle f_N(M'_N), \Phi\rangle +o(1).
$$
By virtue of Theorem \ref{thm:qRpolyslimit}, we may apply this asymptotic relation to
$$
\Phi_N:=\wt\varphi^{qR}_{\la\mid N}(\ccdot;q;s_0,s_1,s_2,s_3), \quad \Phi:=\Phi^{qR}_{\lambda}(\ccdot; q; s_0,s_1, s_2, s_3),
$$ 
with arbitrary $\la\in\Y$. This gives us
\begin{equation}\label{eq3.limit}
\langle M'_N, \wt\varphi^{qR}_{\la\mid N}(\ccdot;q;s_0,s_1,s_2,s_3)\rangle=\langle f_N(M'_N), \Phi^{qR}_{\lambda}(\ccdot; q; s_0,s_1, s_2, s_3)\rangle +o(1).
\end{equation}

On the other hand, we know that
$$
\langle M_N, \varphi^{qR}_{\la\mid N}(\ccdot;q;s_0,s_1,s_2,s_3)\rangle =\de_{\la,\emptyset}.
$$
From this and the very definition of the polynomials $\wt\varphi^{qR}_{\la\mid N}(\ccdot;q;s_0,s_1,s_2,s_3)$ and the measures $M'_N$ we obtain 
$$
\langle M'_N, \wt\varphi^{qR}_{\la\mid N}(\ccdot;q;s_0,s_1,s_2,s_3)\rangle =\de_{\la,\emptyset}.
$$
Next, because of \eqref{eq3.limit}, we deduce
\begin{equation}\label{eq3.limit2}
\langle f_N(M'_N), \Phi^{qR}_{\lambda}(\ccdot; q; s_0,s_1, s_2, s_3)\rangle=\de_{\la,\emptyset}+o(1).
\end{equation}

Recall that the space $\Om^{qR}$ is compact and observe that the image of $\Sym$ under the map $\Sym\to C(\Om^{qR})$ is dense, by the Stone-Weierstrass theorem (because symmetric functions separate points of $\Om^{qR}$). Since the functions $\Phi^{qR}_{\lambda}(\ccdot; q; s_0,s_1, s_2, s_3)$ form a basis of $\Sym$, we conclude from \eqref{eq3.limit2} that the measures $f_N(M'_N)$ weakly converge to a probability measure $M$, which is uniquely characterized by the relations
$$
\langle M, \Phi^{qR}_{\lambda}(\ccdot; q; s_0,s_1, s_2, s_3)\rangle =\de_{\la,\emptyset}, \quad \la\in\Y.
$$
This completes the proof.
\end{proof}

The following two propositions are easy corollaries of the previous results. 

\begin{proposition}
The probability measure $M^{qR}_{q, \zeta, s_0, s_1, s_2, s_3}$ on $\Om^{qR}$ constructed in Theorem \ref{thm:qRorth} is purely atomic.
\end{proposition}

\begin{proof}
Given a configuration $X$ in $\De_{L,R}$, we denote by $X^\circ$ its complement $\De_{L.R}\setminus X$. As usual, we set $K:=L-R$. Let us abbreviate $M:=M^{qR}_{q, \zeta, s_0, s_1, s_2, s_3}$.  We are going to prove that $M$ is concentrated on the countable subset
\begin{equation}
\{X\in \Om^{qR}: |X^\circ|\le K\},
\end{equation}
which will imply the proposition. 

It suffices to prove that for an arbitrary fixed subset $A\subset\De_{L,R}$ of cardinality $K+1$, the set
$$
\Om^{qR}[A]:=\{X\in\Om^{qR}: X^\circ \supseteq A\}
$$
has $M$-measure $0$. 

From the proof of Theorem \ref{thm:qRorth}, we know that $M$ is the weak limit of the measures that we denoted by $f_N(M'_N)$. Each of these measures has finite support. We claim that for $N$ large enough, each configuration from the support of $f_N(M'_N)$ has a nonempty intersection with $A$. Indeed, observe that the grid $\widetilde{\De}_{L+\frac{N-1}{2}, R-\frac{N-1}{2}}$ consists of $K+N$ points, so that each $N$-point configuration on this grid has precisely $K$ holes (that is, unoccupied nodes). If $N$ is large enough, then the  finite grid $f_N(q^{\frac{N-1}2}\widetilde{\De}_{L+\frac{N-1}{2}, R-\frac{N-1}{2}})$ contains $A$, and our claim follows from the very definition of the measures $M'_N$.  

Thus, for $N$ large enough, the set $\Om^{qR}[A]$ has measure $0$ with respect to $f_N(M'_N)$. Since this set is both open and closed, its characteristic function is continuous, so that we may pass to the limit, as $N\to\infty$, and conclude that $\Om^{qR}[A]$ has measure $0$ with respect to $M$. 
\end{proof}

\begin{proposition}
We assume that $(s_0,s_1,s_2,s_3)$ is admissible and we keep to the notation of Theorem \ref{thm:qRorth}. As above, we set $K:=L-R$ and abbreviate $M:=M^{qR}_{q, \zeta, s_0, s_1, s_2, s_3}$, $\Phi^{qR}_\la:=\Phi^{qR}_{\lambda}(\ccdot; q; s_0,s_1, s_2, s_3)$.

The functions $\Phi^{qR}_{\lambda}$ with index $\la\in\Y$, subject to condition $\la_1\le K$, form an orthogonal basis of the real Hilbert space $L^2(\Om^{qR}, M)$.
\end{proposition}

\begin{proof}
Recall (Theorem \ref{thm:qRpolyslimit}) that the elements $\Phi^{qR}_\la\in\Sym$ coincide with limits of $N$-variate $q$-Racah polynomials, which in turn are built from univariate polynomials. Because $(s_0,s_1,s_2,s_3)$ is admissible, the univariate polynomials have real coefficients by virtue of Lemma \ref{positivecoeffs}. Therefore,  the same holds for the $N$-variate polynomials, and hence the symmetric functions $\Phi^{qR}_{\lambda}$ are well defined over the ground field $\R$. This in turn implies that they produce real-valued functions on $\Om^{qR}$. 

The linear span of those functions is dense in $C(\Om^{qR})$ and hence in $L^2(\Om^{qR},M)$, too. Next, our assumptions on the parameters guarantee that Theorem \ref{thm:formalorth} is applicable. It shows that the functions $\Phi^{qR}_\la$, $\la\in\Y$, are pairwise orthogonal in $L^2(\Om^{qR},M)$. Moreover, from \eqref{eqn:hqR} it is seen that the squared norm of $\Phi^{qR}_\la$ is nonzero if and only if $\la_1\le K$. This completes the proof. 
\end{proof}

\section{Big $q$-Jacobi symmetric functions}\label{sect4}

\subsection{Univariate big $q$-Jacobi polynomials}
For more detail about the material of this subsection, see \cite{Ols-2017} and references therein.

The \emph{big $q$-Jacobi $q$-difference operator} is defined by
\begin{equation*}
D^{bqJ}=A_+(x)(T_q-1)+A_-(x)(T_{q^{-1}}-1),
\end{equation*}
where the coefficients $A_\pm(x)$ are
\begin{equation*}
\begin{gathered}
A_+(x):=\frac{cdq}{ab}\left(1-\frac1{cx}\right)\left(1-\frac1{dx}\right),\\
A_-(x):=\left(1-\frac{q}{ax}\right)\left(1-\frac{q}{bx}\right).
\end{gathered}
\end{equation*}
Here $x$ and  $(a,b,c,d)$ can be initially thought of as a complex variable and generic complex parameters, respectively. A bit later we will impose appropriate constraints on them.

A direct verification shows that
\begin{multline*}
D^{bqJ} x^n=-(q^{-n}-1)\left(\frac{cdq^{n+1}}{ab}-1\right) x^n\\
+\left(-\frac q{ab}(c+d)(q^n-1)-\frac q{ab}(a+b)(q^{-n}-1)\right)x^{n-1}\\
+\left(\frac q{ab}(q^n-1)+\frac{q^2}{ab}(q^{-n}-1)\right) x^{n-2}.
\end{multline*}
Note that the coefficients in front of $x^{n-1}$ and $x^{n-2}$ vanish for $n=0$  and for $n=0,1$, respectively. It follows that $D^{bqJ}$ preserves the space $\C[x]$ together with its natural filtration. Since the quantities $-(q^{-n}-1)\left(\frac{cdq^{n+1}}{ab}-1\right)$, $n=0,1,2,\dots$,  are pairwise distinct (for generic parameters!), there exist monic polynomials
$$
\varphi_n^{bqJ}=\varphi_n^{bqJ} \left( x ; q; a, b, c, d \right), \quad n=0,1,2,\dots,
$$
with $\deg \varphi^{bqJ}_n = n$, which are eigenfunctions of $D^{bqJ}$:
$$
D^{bqJ}\varphi_n^{bqJ}=-(q^{-n}-1)\left(\frac{cdq^{n+1}}{ab}-1\right)\varphi^{bqJ}_n.
$$
Moreover, their coefficients are rational functions of the parameters. These properties determine the polynomials uniquely.

The  polynomials $\varphi_n^{bqJ}$ are called the \emph{big $q$-Jacobi polynomials}. They are symmetric with respect to transpositions $a\leftrightarrow b$ and $c\leftrightarrow d$.  Their hypergeometric representation has the form
\begin{equation}\label{eq4.A1}
\varphi_n^{bqJ} \left( x ; q; a, b, c, d \right) =
\frac{\left( \dfrac{cq}{a}, \dfrac{cq}{b} ; q\right)_n}{c^{n}\left( \dfrac{cd}{ab}q^{n+1} ; q \right)_n}\
{}_3\phi_2\left[\left.\begin{matrix} q^{-n}, \; \dfrac{cd}{ab}q^{n+1}, \; cx \\ \dfrac{cq}{a},\qquad\dfrac{cq}{b} \end{matrix}\,\right|q\right].
\end{equation}

\begin{definition}\label{jacobiadmissible}
We say that a quadruple $(a,b,c,d)$ is \emph{admissible} if $a>0>b$, and the pair $(c, d)$ satisfies one of the following conditions:

\begin{itemize}
	\item $d=\bar c\in\C\setminus\R$;

	\item $a q^m< c,d<aq^{m-1}$ for some $m\in\Z$;

	\item $b q^{m-1} <c,d< b q^m$ for some $m\in\Z$.
\end{itemize}
\end{definition}

For the sequel, it is important to notice that these three conditions are invariant with respect to the homothety with ratio $q^{\pm1}$.

In what follows, we assume (unless otherwise stated) that $(a,b,c,d)$ is admissible. Then the polynomials $\varphi_n^{bqJ} \left( x ; q; a, b, c, d \right)$ are orthogonal on the grid
$$
\De_{a, b} := \{ b^{-1}q, b^{-1}q^2, b^{-1}q^3, \ldots \}  \sqcup \{ \ldots, a^{-1}q^3, a^{-1}q^2, a^{-1}q \}
$$
with respect to the weight function
$$
w^{bqJ}(x):=\const |x|\frac{(ax;q)_\infty(bx;q)_\infty}{(cx;q)_\infty(dx;q)_\infty},\quad \const>0.
$$

\subsection{Construction of the big $q$-Jacobi symmetric functions}

\begin{definition}\label{df:bqJfunction}
For $\lambda\in\Y$, the \textit{big $q$-Jacobi symmetric function} $\Phi^{bqJ}_{\lambda}(\ccdot; q; a, b, c, d)\in\Sym$ is defined by
\begin{equation*}
\Phi^{bqJ}_{\lambda}(X; q; a, b, c, d) := \sum_{\mu\subseteq\lambda}
\sigma^{bqJ}(\lambda, \mu; q; a, b, c, d) I^A_{\mu}(Xc; q),
\end{equation*}
where $I^A_{\mu}(\ccdot; q) \in \Sym$ are defined in Section $\ref{sec:IA}$ and, for $\mu\subseteq\lambda$,
\begin{multline*}
\sigma^{bqJ}(\lambda, \mu; q; a, b, c, d) :=\\
c^{-|\lambda|} \det
\left[ \frac{\left( \frac{c}{a}q^{\mu_k + 2 - k}, \frac{c}{b}q^{\mu_k + 2 - k} ; q\right)_{\lambda_j - j - \mu_k + k}}
{q^{(\mu_k + 1 - k)(\lambda_j - j - \mu_k + k)}
\left( q, \frac{cd}{ab}q^{\lambda_j + \mu_k + 3 - j -  k}; q \right)_{\lambda_j - j - \mu_k + k}} \right]_{j, k = 1}^{\ell(\lambda)}.
\end{multline*}
\end{definition}

(The above definition agrees with that given in \cite[Section 3]{Ols-2017} after a minor correction: the factor $(-1)^{|\la|-|\mu|}$ in \cite[(3.4)]{Ols-2017} should be removed; it arose by mistake, from the unnecessary factor $(-1)^{\ell-m}$ in the unnumbered display after \cite[(3.5)]{Ols-2017}.)

The $N$-variate big $q$-Jacobi polynomials are defined in accordance with the general recipe \eqref{eq1.B}; we would like, however, that two of the parameters $c, d$ depend on $N$ as follows:
\begin{multline*}
\varphi^{bqJ}_{\la | N}(x_1, \ldots, x_N; q; a, b, cq^{1-N}, dq^{1-N}) :=\\
\frac{\det\left[ \varphi^{bqJ}_{\la_i + N - i}(x_j; q; a, b, cq^{1-N}, dq^{1-N}) \right]_{i, j = 1}^N}{V(x_1, \ldots, x_N)}; \ \la\in\Y(N), \ N = 1, 2, \ldots.
\end{multline*}

\begin{theorem}\cite[Theorem 3.4]{Ols-2017}\label{thm:bqJpolyslimit}
For any $\lambda\in\Y$, we have 
\begin{equation*}
\varphi_{\lambda | N}^{bqJ}(\ccdot; q; a, b, cq^{1-N}, dq^{1-N}) \rightarrow \Phi^{bqJ}_{\lambda}(\ccdot; q; a, b, c, d),
\end{equation*}
in the sense of Definition $\ref{def2.convergence}$.
\end{theorem}

\subsection{Orthogonality measures}

Let $\Omega^{bqJ}$ be the set of all point configurations on $\De_{a, b}$, equipped with the topology coming from its identification with $\{0, 1\}^{\De_{a, b}}$. For $\la\in\Y$, we set
\begin{equation}\label{eqn:hbqJ}
h_{\la}^{bqJ}(q; a, b, c, d) := q^{2 \{ n(\la') - n(\la) \}}
\frac{ \left(\frac{c q}{a}, \frac{c q}{b}, \frac{d q}{a}, \frac{d q}{b}; q\right)_{\la}}{ \left( \frac{cdq^2}{ab}; q \right)_{\widehat{\la}}} \left( \frac{cd q^3}{a^2 b^2} \right)^{|\la|}.
\end{equation}
Recall that $\widehat{\la}$ denotes the partition $(2\la_1, 2\la_1, 2\la_2, 2\la_2, 2\la_3, \dots)$  whose Young diagram is obtained from the Young diagram of $\la$ by replacing each square with a $2\times 2$ square.

\begin{theorem}\cite[Theorem 4.1]{Ols-2017}\label{thm:bqJorth}
Assume that $(a, b, c, d)$ is admissible in the sense of Definition $\ref{jacobiadmissible}$.

{\rm(i)} There exists a unique probability measure $M^{bqJ}_{q, a, b, c, d}$ on $\Omega^{bqJ}$ such that the big $q$-Jacobi symmetric functions  $\{ \Phi^{bqJ}_{\la} = \Phi_{\la}^{bqJ}(\ccdot; q; a, b, c, d) \}_{\la\in\Y}$ form an orthogonal basis in the Hilbert space $L^2(\Omega^{bqJ}, M^{bqJ}_{q, a, b, c, d})$.

{\rm(ii)} The squared norms of the these functions are given by
$$
\left(\Phi_{\lambda}^{bqJ}, \Phi_\la^{bqJ} \right)_{L^2 \left(\Omega^{bqJ}, M^{bqJ}_{q, a, b, c, d} \right)}
=h_{\la}^{bqJ}(q; a, b, c, d),
$$
where $h_{\la}^{bqJ}(q; a, b, c, d)$ is defined in $(\ref{eqn:hbqJ})$.
\end{theorem}

\begin{remark}\label{rem4.A}
As in the case of $q$-Racah symmetric functions, in the definition of the big $q$-Jacobi symmetric functions $\Phi_{\lambda}^{bqJ}$, the assumptions on the parameters $a,b,c,d$ may be relaxed: all of them they may be complex numbers subject to mild (generic) constraints. Then we obtain what we call `formal orthogonality': the coefficient of $\Phi^{bqJ}_{\emptyset} = 1$ in the expansion of $\Phi^{bqJ}_{\la}\Phi^{bqJ}_{\mu}$ (in the basis $\{\Phi^{bqJ}_{\nu}\}_{\nu}$) is $\delta_{\la, \mu}h^{bqJ}_{\la}(q; a, b, c, d)$.
\end{remark}

\subsection{Limit transition: $q$-Racah $\rightarrow$ big $q$-Jacobi.}

In the article \cite{Koo-2011}, Koornwinder observed that the limit transition `$q$-Racah $\rightarrow$ big $q$-Jacobi' that is well-known in the literature, \cite[Section 4.6]{KS}, has one defect: the orthogonality of the $q$-Racah polynomials is lost when they are renormalized to take the limit to big $q$-Jacobi polynomials.
The main result of \cite{Koo-2011} is a new limit transition such that the $q$-Racah polynomials still form finite families of orthogonal polynomials as they approach the big $q$-Jacobi polynomials.
Moreover, the orthogonality measures for the renormalized $q$-Racah polynomials converge weakly to the orthogonality measures of the big $q$-Jacobi polynomials. Below we use this result in a slightly modified form. 

\begin{proposition}\label{prop4.A}
The following limit holds for all $n\in\Z_{\ge0}$:
\begin{equation}\label{koornwinderlimit}
\varphi^{bqJ}_n(x; q; a, b, c, d) = \lim_{\eps \rightarrow 0}{\eps^n\varphi^{qR}_n(x\eps^{-1}; q; \eps c, \eps d, -q/(\eps a), -q/(\eps b))}.
\end{equation}
\end{proposition}

\begin{proof}
This can be deduced from the $q$-hypergeometric representation of the polynomials in question (formulas 
\eqref{eqn:qRpoly} and \eqref{eq4.A1} above) by a short computation, similar to that in \cite[Theorem 1]{Koo-2011}. An alternative way is to pass to the limit in terms of the corresponding $q$-difference operators. 
\end{proof}

\begin{lemma}\label{lemma4.A}
Fix an arbitrary quadruple $(a,b,c,d)$ which is admissible in the sense of Definition \ref{jacobiadmissible}. For  $M\in\Z_{\ge0}$, set
$$
\eps:=q^{M+1}/\sqrt{-ab}, \quad \zeta:=\sqrt{-b/a}, \quad L:=M, \quad R:=-M.
$$
Then the quadruple 
$$
(s_0,s_1,s_2,s_3):=(\eps c, \eps d, -q/(\eps a), -q/(\eps b)),
$$ 
entering the right-hand side of \eqref{koornwinderlimit}, is admissible in the sense of Definition \ref{df:qRfunction}, with the same parameters $\zeta,L,R$ as indicated above. 
\end{lemma}

\begin{proof}
Recall that $a>0>b$, so that $\zeta>0$. Obviously, $s_2=-\zeta q^R$ and $s_3=\zeta^{-1} q^{-L}$, as required. Next, it is directly checked that the three conditions on $(c,d)$ in Definition \ref{df:bqJfunction} are exactly translated into the three conditions on $(s_0,s_1)$ in Definition \ref{df:qRfunction}.
\end{proof}

\begin{theorem}\label{thm:qRbqJ}
Let $(a,b,c,d)$ be admissible in the sense of Definition \ref{jacobiadmissible}. For any $\lambda\in\Y$, one has\begin{equation}\label{eqn:qRbqJ}
\Phi_{\lambda}^{bqJ}(X; q; a, b, c, d) =
\lim_{\eps\rightarrow 0}{ \eps^{|\lambda|}\Phi_{\lambda}^{qR}\left( X \eps^{-1}; q; \eps c, \eps d, -\frac{q}{\eps a}, -\frac{q}{\eps b} \right) }.
\end{equation}
Here we assume that $\eps$ goes to $0$ along the grid $q^{\Z_{\geq 1}}/\sqrt{-ab}$, as in Lemma \ref{lemma4.A}, and the limit holds in the finite-dimensional subspace of symmetric functions of degree $\leq |\lambda|$.
\end{theorem} 

Note that, by virtue of Lemma \ref{lemma4.A}, the quadruple of the parameters on the right-hand side is admissible in the sense of Definition \ref{df:qRfunction}. This is the only reason why we want $(a,b,c,d)$ to be admissible and $\eps$ range over a special grid.
The argument below is formal and does not really use this assumption. 

\begin{proof}
By comparing the formulas in Definitions \ref{df:qRfunction} and \ref{df:bqJfunction}, it suffices to prove
\begin{equation}\label{eqn:limitsRtobqJ}
\begin{gathered}
\lim_{\eps \rightarrow 0}{ \eps^{|\lambda| - |\mu|} \sigma^{qR} \left( \lambda, \mu; q; \eps c, \eps d, -\frac{q}{\eps a}, -\frac{q}{\eps b} \right) }
 = c^{|\mu|} \sigma^{bqJ}(\lambda, \mu; q; a, b, c, d),\\
\lim_{\eps\rightarrow 0}{ \eps^{|\mu|} I_{\mu}^{BC}(X\eps^{-1}; q; \eps c) } = c^{-|\mu|} I^A_{\mu}(Xc; q),
\end{gathered}
\end{equation}
for all $\mu\subseteq\lambda$.
The second limit in $(\ref{eqn:limitsRtobqJ})$ is exactly that of Proposition $\ref{BCtoA}$, with $X$ replaced by $Xc$ and $\eps$ replaced by $c\eps$. Let us now prove the first limit.
Fix $\mu\subseteq\lambda$ and let us estimate the terms in formula $(\ref{eqn:sigmaqR})$ after setting $s_0 := \eps c,\ s_1 := \eps d,\ s_2 := -q/(\eps a),\ s_3 := -q/(\eps b)$:

\smallskip

	$\bullet$ $(-s_0 s_1q^{\mu_k+1-k}; q)_{\lambda_j - j - \mu_k + k} = (-cd q^{\mu_k+1-k}\eps^2; q)_{\lambda_j - j - \mu_k + k} = 1 + O(\eps^2)$,
	
	$\bullet$ $(-s_0 s_2q^{\mu_k+1-k}; q)_{\lambda_j - j - \mu_k + k} = \left( \frac{c}{a}q^{\mu_k + 2 - k} ; q\right)_{\lambda_j - j - \mu_k + k}$,
	
	$\bullet$ $(-s_0 s_3q^{\mu_k+1-k}; q)_{\lambda_j - j - \mu_k + k} = \left( \frac{c}{b}q^{\mu_k + 2 - k} ; q\right)_{\lambda_j - j - \mu_k + k}$,
	
	$\bullet$ $(s_0 s_1 s_2 s_3q^{\lambda_j + \mu_k + 1 - j - k}; q)_{\lambda_j - j - \mu_k + k} =
\left( \frac{cd}{ab}q^{\lambda_j + \mu_k + 3  -  j -  k} ; q\right)_{\lambda_j - j - \mu_k + k}$,
	
	$\bullet$ $s_0^{|\mu| - |\lambda|} = \eps^{|\mu| - |\lambda|} c^{|\mu| - |\lambda|}$.

\smallskip

By putting all these items together, we obtain the first limit in $(\ref{eqn:limitsRtobqJ})$.
\end{proof}

We are going to show that the limit transition in Theorem \ref{thm:qRbqJ} is consistent with the orthogonality measures. For the precise formulation, we need  a little preparation. 

We keep to the notation
\begin{equation}\label{eq4.A}
\begin{gathered}
\eps:=q^{M+1}/\sqrt{-ab}, \quad \zeta:=\sqrt{-b/a}, \quad L:=M, \quad R:=-M, \\
(s_0,s_1,s_2,s_3):=(\eps c, \eps d, -q/(\eps a), -q/(\eps b)),
\end{gathered}
\end{equation}
introduced in Lemma \ref{lemma4.A}. As above, we assume that $(a,b,c,d)$ is admissible in the sense of Definition  \ref{df:bqJfunction} and is fixed, while $M$ ranges over $\Z_{\ge0}$. Thus, $\eps$ and $(s_0,s_1,s_2,s_3)$ vary together with $M$. 

Since (as was already noted above) the quadruple $(s_0,s_1,s_2,s_3)$ is admissible in the sense of Definition  \ref{df:qRfunction}, Theorem \ref{thm:qRorth} tells us about the existence of a probability measure $M^{qR}_{q,\zeta,s_0,s_1,s_2,s_3}$. This measure serves as the orthogonality measure for the corresponding $q$-Racah symmetric functions and lives on the space of configurations on the grid
$$
\De^{qR}_{M, -M} := \{ -\zeta^{-1}q^{-M}, -\zeta^{-1}q^{1-M}, \ldots\} \sqcup \{\ldots, \zeta q^{1-M}, \zeta q^{-M} \}.
$$

On the other hand, we know from Theorem \ref{thm:bqJorth} that the big $q$-Jacobi symmetric functions with parameters $(a,b,c,d)$ are orthogonal with respect to a probability measure $M^{bqJ}_{q, a, b, c, d}$, which lives on configurations on the grid 
$$
\Delta^{bqJ}_{a, b} := \{ b^{-1}q, b^{-1}q^2, \ldots \} \sqcup \{\ldots, a^{-1}q^2, a^{-1}q\}.
$$

Now we observe that homothety $x\mapsto x\eps$ induces a bijection $\De^{qR}_{M, -M} \to \Delta^{bqJ}_{a, b}$, because $b^{-1}q=-\zeta^{-1}q^{-M}\eps$ and $a^{-1}q=\zeta q^{-M}\eps$, as is seen from \eqref{eq4.A}. 

This homothety extends to a bijection, $X\mapsto X\eps$, between the two spaces of point configurations. This  allows us to define a rescaled version $\wt M^{qR}_{q,\zeta,s_0,s_1,s_2,s_3}$ of the measure $M^{qR}_{q,\zeta,s_0,s_1,s_2,s_3}$ --- its pushforward under the map $X\mapsto X\eps$. Thus,  $\wt M^{qR}_{q,\zeta,s_0,s_1,s_2,s_3}$ and $M^{bqJ}_{q,a,b,c,d}$ are defined on the same space.  

\begin{theorem}
Let $M^{bqJ}_{q,a,b,c,d}$ be the orthogonality measure for the big $q$-Jacobi symmetric functions with fixed admissible parameters $(a,b,c,d)$. Next, let $M$ range over $\Z_{\ge0}$, and let $\zeta$ and $(s_0,s_1,s_2,s_3)$ be defined by \eqref{eq4.A}. Finally, let $\wt M^{qR}_{q,\zeta,s_0,s_1,s_2,s_3}$ be the rescaled version of the $q$-Racah orthogonality measure $M^{qR}_{q,\zeta,s_0,s_1,s_2,s_3}$,  as defined above. Recall that  $\wt M^{qR}_{q,\zeta,s_0,s_1,s_2,s_3}$ depends on $M$.

As $M\to\infty$, the measures $\wt M^{qR}_{q,\zeta,s_0,s_1,s_2,s_3}$ weakly converge to the measure $M^{bqJ}_{q, a, b, c, d}$.
\end{theorem}

\begin{proof}
This is a direct consequence of Theorem \ref{thm:qRbqJ}. Indeed, set
$$
\wt\Phi^{qR}_\la(X;q; s_0,s_1,s_2,s_3):=\eps^{|\la|}\Phi^{qR}_\la(X\eps^{-1};q; s_0,s_1,s_2,s_3);
$$
this is the prelimit expression on the right-hand side of \eqref{eqn:qRbqJ}. The measure $\wt M^{qR}_{q,\zeta,s_0,s_1,s_2,s_3}$ is uniquely characterized by the relations
\begin{equation}\label{eq4.B}
\langle \wt M^{qR}_{q,\zeta,s_0,s_1,s_2,s_3}, \; \wt\Phi^{qR}_\la(\ccdot;q; s_0,s_1,s_2,s_3)\rangle=\de_{\la,\emptyset}, \quad \la\in\Y,
\end{equation}
as is seen from the very definition of $\wt M^{qR}_{q,\zeta,s_0,s_1,s_2,s_3}$ and the similar characterization of the measure $M^{qR}_{q,\zeta,s_0,s_1,s_2,s_3}$ in Theorem \ref{thm:qRorth}. 

On the other hand, Theorem \ref{thm:qRbqJ} implies that, as $M\rightarrow\infty$, 
\begin{equation}\label{eq4.B2}
\wt\Phi^{qR}_\la(X;q; s_0,s_1,s_2,s_3) \to \Phi_{\lambda}^{bqJ}(X; q; a, b, c, d), \quad \la\in\Y, 
\end{equation}
in the finite-dimensional space of symmetric functions of degree $\leq |\la|$.
Therefore, since $\Omega^{bqJ}$ is the set of point configurations on the double $q$-grid $\De^{bqJ}_{a, b}$, the limit \eqref{eq4.B2} also holds uniformly for $X$ belonging to $\Omega^{bqJ}$.
Then \eqref{eq4.B} and \eqref{eq4.B2} show that there exists a weak limit
$$
\wt M:=\lim_{M\rightarrow\infty} \wt M^{qR}_{q,\zeta,s_0,s_1,s_2,s_3},
$$
which satisfies 
$$
\langle \wt M, \; \Phi^{bqJ}_\la(\ccdot;q; a,b,c,d)\rangle=\de_{\la,\emptyset}, \quad \la\in\Y.
$$

Finally, applying Theorem \ref{thm:bqJorth} (i), we conclude that $\wt M=M^{bqJ}_{q,a,b,c,d}$. 
\end{proof}

\section{$q$-Meixner symmetric functions}\label{sect5}

\subsection{Univariate $q$-Meixner polynomials}

We define the \emph{$q$-Meixner difference operator} by
$$
D^{qM} := A_+(x)(T_q - 1) + A_-(x)(T_{q^{-1}} - 1),
$$
with
$$
A_+(x) := \left(c - \frac{1}{x}\right)\left(d - \frac{1}{x}\right), \ A_-(x) := -\frac{1}{x} \left(a - \frac{q}{x}\right).
$$
The variable $x$ is complex, and the parameters $a, c, d$ are generic complex numbers. 

This operator is a degeneration of the big $q$-Jacobi operator $D^{bqJ}$ defined in the beginning of Section \ref{sect4}. Namely, 
$$
D^{qM}=\lim_{b\to0}\left(\frac{ab}q D^{bqJ}\right).
$$

One can easily verify
\begin{gather*}
D^{qM}x^n = cd(q^n - 1)x^n + \left( -(c+d)(q^n - 1) - a(q^{-n} - 1) \right)x^{n-1}\\
+ \left( (q^n - 1) + q(q^{-n} - 1) \right)x^{n-2}, \ n = 0, 1, 2, \dots,
\end{gather*}
which shows immediately that $D^{qM}$ preserves $\C[x]$ and the corresponding filtration by degree.
For nonzero $c, d$, the values $cd(q^n - 1)$, $n = 0, 1, 2, \ldots$, are all distinct.
Then, for each $n = 0, 1, \ldots$, there is a unique monic polynomial $\varphi^{qM}_n = \varphi^{qM}_n(x; q; a, c, d)$ of degree $n$ such that
$$D^{qM}\varphi^{qM}_n = cd(q^n - 1)\varphi^{qM}_n.$$
The polynomials $\varphi^{qM}_n$ are the \textit{$q$-Meixner polynomials}; they are determined by the properties previously mentioned, at least for generic $a, c, d$.
Their hypergeometric representation is:
\begin{equation}\label{eqn:qMdef}
\varphi_n^{qM} \left(x ; q; a, c, d \right) =  \left( \frac{cq}{a} ; q \right)_n \left( \frac{a}{cdq^n} \right)^n
{}_2\phi_1\left[\left.\begin{matrix} q^{-n}, \; cx \\ \frac{cq}{a} \end{matrix}\,\right| \frac{d}{a}q^{n+1} \right], \ n = 0, 1, \dots.
\end{equation}
Our chosen normalization for the $q$-Meixner polynomials was so that the following degeneration big $q$-Jacobi $\rightarrow$ $q$-Meixner holds:
$$\lim_{b \rightarrow 0}{\varphi^{bqJ}_n(x; q; a, b, c, d)} = \varphi^{qM}_n(x; q; a, c, d).$$
In the literature, the $q$-Meixner polynomials are typically normalized differently.
For example, with respect to the polynomials $M_n(x; B, C; q)$ in \cite[3.13]{KS}, we have
\begin{equation*}
\varphi^{qM}_n(x; q; a, c, d) = \left( \frac{cq}{a} ; q \right)_n \left( \frac{a}{cdq^n} \right)^n
M_n(cx; c/a, -a/d; q).
\end{equation*}
In terms of the polynomials $m_n(x; B, C; q)$ in \cite{GK}, we have
\begin{equation*}
\varphi^{qM}_n(x; q; a, c, d) = \left( \frac{cq}{a}, \frac{dq}{a} ; q \right)_n \left( \frac{a}{cdq^n} \right)^n m_n(-ax/q; dq/a, cq/a; q).
\end{equation*}

The $q$-Meixner polynomials are related to an indeterminate moment problem, \cite{C}, and hence admit a variety of orthogonality measures, \cite{Ak}. Some concrete examples can be found in \cite[Section 3.13]{KS}  (for the polynomials $M_n(x; B, C; q)$) and in \cite{GK} (for the polynomials $m_n(x; B, C; q)$).
Below, we use the weight functions from \cite{GK} with a minor modification.

\begin{definition}\label{meixneradmissible}
We say that a triple $(a,c,d)$ is \emph{admissible} if $a>0$, and the pair $(c, d)$ satisfies one of the following conditions:
\begin{itemize}
\item $d=\bar c\in\C\setminus\R$;
\item $aq^m < c,d<aq^{m-1}$, for some $m\in\Z$;
\item $c, d<0$ and $q<cd^{-1}<q^{-1}$.
\end{itemize}
\end{definition}

Our constraints on the parameters are similar to those in \cite[Corollary 2.4]{GK}; some difference is caused by the fact that, in our construction, we will require that $(a, cq^{1-N}, dq^{1-N})$ is an admissible triple whenever $(a, c, d)$ is admissible.

In the sequel, we assume that $(a, c, d)$ is an admissible triple.
To any such triple, we attach a one-parameter family of grids of the form 
\begin{equation}\label{eq5.A}
\De_{a,\be} :=\be^{-1} q^\Z \,\sqcup\, a^{-1} q^{\Z_{\ge1}} = \{ \ldots, \be^{-1} q^{-1}, \be^{-1}, \be^{-1} q, \ldots \}  \sqcup \{ \ldots, a^{-1}q^2, a^{-1}q \},
\end{equation}
where $\be<0$ may be arbitrary, unless $c,d<0$ (the third option above), in which case we additionally require that $\be q^{m-1}<c,d< \be q^m$, for some $m\in\Z$.

Obviously, $\De_{a, \be}$ does not change if $\be$ is multiplied by $q^{\pm1}$. Note also that, unlike the grids for $q$-Racah polynomials and big $q$-Jacobi polynomials, the grids $\De_{a, \be}$ are not bounded.

The polynomials $\varphi^{qM}_n(x; q; a, c, d)$ are orthogonal on $\De_{a,\be}$ 
with respect to the weight
$$
w^{qM}(x) := \const|x|\frac{(ax; q)_{\infty}}{(cx, dx; q)_{\infty}}, \quad  \const > 0.
$$
Note that $w^{qM}(x) >0$, for all $x\in\De_{a,\be}$.

\subsection{Construction of the $q$-Meixner symmetric functions}

\begin{definition}\label{df:qMfunction}
For any partition $\lambda$, we define the \textit{$q$-Meixner symmetric function} $\Phi^{qM}_{\lambda}(\ccdot; q; a, c, d)\in\Sym$ by
\begin{equation*}
\Phi^{qM}_{\lambda}(X; q; a, c, d) := \sum_{\mu\subseteq\lambda}
\sigma^{qM}(\lambda, \mu; q; a, c, d) I^A_{\mu}(Xc; q),
\end{equation*}
where $I^A_{\mu}(\ccdot; q) \in \Sym$ are defined in Section $\ref{sec:IA}$ and, for $\mu\subseteq\lambda$,
\begin{multline*}
\sigma^{qM}(\lambda, \mu; q; a, c, d) :=
c^{-|\lambda|} \left( \frac{a}{qd} \right)^{|\lambda| - |\mu|}\\
\times q^{2\{n(\mu') - n(\mu) - n(\lambda') + n(\lambda)\}}
\det \left[ \frac{\left( \frac{c}{a}q^{\mu_k + 2 - k} ; q\right)_{\lambda_j - j - \mu_k + k}}
{\left( q ; q \right)_{\lambda_j - j - \mu_k + k}} \right]_{j, k = 1}^{\ell(\lambda)}.
\end{multline*}
\end{definition}

We can define the multivariate $q$-Meixner polynomials in accordance with \eqref{eq1.B}.
However, we will want the two parameters $(c, d)$ to depend on $N$ as follows:
\begin{equation*}
\varphi^{qM}_{\la\mid N}(x_1, \ldots, x_N; q; a, cq^{1-N}, dq^{1-N})
:= \frac{\det\left[ \varphi^{qM}_{\la_i + N - i}(x_j; q; a, cq^{1-N}, dq^{1-N}) \right]_{i, j=1}^N}{V(x_1, \ldots, x_N)}.
\end{equation*}

\begin{theorem}\label{thm:qMpolyslimit}
For any $\lambda\in\Y$, as $N\to\infty$,
\begin{equation*}
\varphi_{\lambda | N}^{qM}(\ccdot; q; a, cq^{1-N}, dq^{1-N}) \rightarrow \Phi^{qM}_{\lambda}(\ccdot; q; a, c, d),
\end{equation*}
in the sense of Definition $\ref{def2.convergence}$.
\end{theorem}

\begin{proof}
The proof is very similar to the proof of Theorem $\ref{thm:bqJpolyslimit}$ in \cite{Ols-2017} (see also the proof of the more complicated Theorem $\ref{thm:qRpolyslimit}$ above). Thus we only sketch the steps.

From $(\ref{eqn:qMdef})$, for any $\ell, N\in\N$ we have that $\varphi^{qM}_{\ell}(x; q; a, cq^{1-N}, dq^{1-N})$ equals
$$\sum_{m = 0}^{\ell}{ q^{2(N-1)(\ell - m)} \frac{(q; q)_{\ell}}{(q; q)_m} c^{-\ell} \left( \frac{a}{d} \right)^{\ell-m}
q^{m^2 - \ell^2}
\frac{\left( \frac{cq^{m+2-N}}{a}; q \right)_{\ell-m}}{(q; q)_{\ell-m}} }
(cx | q^{N-1}, q^{N-2}, \ldots)^m.$$
This expansion can be generalized to many variables, by the Cauchy-Binet formula; the result is
$$\varphi^{qM}_{\lambda\mid N}(x_1, \ldots, x_N; q; a, cq^{1-N}, dq^{1-N}) = \sum_{\mu\subseteq\lambda}{\sigma^{qM}_N(\lambda, \mu; q; a, c, d) I^A_{\mu\mid N}(cx_1, \ldots, cx_N; q) },$$
where
\begin{gather*}
\sigma^{qM}_N(\lambda, \mu; q; a, c, d)  = c^{\frac{N(N-1)}{2}} q^{2(N-1)(|\la| - |\mu|)}
\prod_{i = 1}^N{\frac{(q; q)_{\lambda_i+N-i}}{(q; q)_{\mu_i+N-i}}}\\
\times\det \left[ c^{-(\lambda_j+N-j)} \left( \frac{a}{d} \right)^{\la_j - j - \mu_k + k} 
q^{(\mu_k + N - k)^2 - (\lambda_j + N - j)^2}
\frac{\left( \frac{c}{a}q^{\mu_k + 2 - k} ; q\right)_{\lambda_j - j - \mu_k + k}}
{\left( q ; q \right)_{\lambda_j - j - \mu_k + k}} \right]_{j, k = 1}^{N}.
\end{gather*}
The result then follows after checking $\sigma^{qM}_N(\lambda, \mu; q; a, c, d) \xrightarrow[N \to \infty]{}\sigma^{qM}(\lambda, \mu; q; a, c, d)$, and using Proposition $\ref{prop2.C}$.
\end{proof}

\subsection{Limit transition: big $q$-Jacobi $\rightarrow$ $q$-Meixner.}
As above, we assume that $(a,c,d)$ is admissible in the sense of Definition \ref{meixneradmissible}. Let $b<0$ be an extra parameter. We want the quadruple $(a,b,c,d)$ to be admissible in the sense of Definition \ref{jacobiadmissible}, so that in the case when $c,d<0$, we impose the additional constraint that $b q^{m-1}<c,d< b q^m$, for some $m\in\Z$. Due to the condition $q<cd^{-1}<q^{-1}$, this constraint will allow to take $b$ arbitrarily small, which is needed in the following theorem.

\begin{theorem}\label{thm:bqJqM}
For any $\lambda\in\Y$,
\begin{equation*}
\lim_{b\rightarrow 0}{ \Phi_{\lambda}^{bqJ}\left(\ccdot; q; a, b, c, d \right) }
= \Phi_{\lambda}^{qM}(\ccdot; q; a, c, d),
\end{equation*}
in the finite-dimensional space of symmetric functions of degree $\le|\la|$.
\end{theorem}
\begin{proof}
From Definitions $\ref{df:bqJfunction}$ and $\ref{df:qMfunction}$, it suffices to show that
$$
\lim_{b\rightarrow 0}{\sigma^{bqJ}(\lambda, \mu; q; a, b, c, d)} = \sigma^{qM}(\lambda, \mu; q; a, c, d), \quad \mu\subseteq\lambda.
$$
For $1 \leq j, k \leq \ell(\la)$, we have
$$
\lim_{b\rightarrow 0}{\frac{\left( \frac{c}{b}q^{\mu_k+2-k} ; q\right)_{\la_j - j - \mu_k + k}}{\left( \frac{cd}{ab}q^{\la_j+\mu_k+3-j-k} ; q\right)_{\la_j - j - \mu_k + k}}} = 
\left( \frac{a}{d}q^{j-1-\la_j} \right)^{\lambda_j-j-\mu_k+k}.
$$
Therefore the limit $\lim_{b\rightarrow 0}{\sigma^{bqJ}(\lambda, \mu; q; a, b, c, d)}$ exists and equals
\begin{equation}\label{eqn:limitb0}
c^{-|\lambda|} \left( \frac{a}{d} \right)^{|\la| - |\mu|} \det
\left[ \frac{  q^{(j-1-\la_j)(\lambda_j-j-\mu_k+k)}\left( \frac{c}{a}q^{\mu_k + 2 - k} ; q\right)_{\lambda_j - j - \mu_k + k}}
{q^{(\mu_k + 1 - k)(\lambda_j - j - \mu_k + k)}
\left( q; q \right)_{\lambda_j - j - \mu_k + k}} \right]_{j, k = 1}^{\ell(\lambda)}.
\end{equation}
Simple calculations show
$$\frac{q^{(j-1-\la_j)(\la_j-j-\mu_k+k)}}{q^{(\mu_k+1-k)(\la_j-j-\mu_k+k)}} = q^{(j-\la_j+k-\mu_k-2)(\la_j-j-\mu_k+k)} = q^{(k-\mu_k-1)^2 - (j-\la_j-1)^2}$$
and
\begin{gather*}
\sum_{k = 1}^{\ell(\la)}{(k-\mu_k-1)^2} - \sum_{j = 1}^{\ell(\la)}{(j-\la_j-1)^2} =
|\mu| - |\la| + 2\{n(\mu') - n(\mu) - n(\la') + n(\la)\}
\end{gather*}
by using $n(\la) = \sum_{i=1}^{\ell(\la)}{(i-1)\la_i}$, $n(\la') = \sum_{i=1}^{\ell(\la)}{{\la_i \choose 2}}$. Therefore, $(\ref{eqn:limitb0})$ equals
$$
c^{-|\lambda|} \left( \frac{a}{d} \right)^{|\la| - |\mu|} q^{|\mu| - |\la| + 2\{ n(\mu') - n(\mu) - n(\la') + n(\la) \}}
\det\left[ \frac{  \left( \frac{c}{a}q^{\mu_k + 2 - k} ; q\right)_{\lambda_j - j - \mu_k + k}}
{\left( q; q \right)_{\lambda_j - j - \mu_k + k}} \right]_{j, k = 1}^{\ell(\lambda)},
$$
and this last expression is exactly $\sigma^{qM}(\la, \mu; q; a, c, d)$.
\end{proof}

\subsection{Formal orthogonality}
We abbreviate $\Phi^{qM}_{\la} := \Phi^{qM}_{\la}(\ccdot; q; a, c, d)$. Associated with the basis $\{ \Phi^{qM}_{\la} \}_{\la\in\Y}$ of $\Sym$ is the moment functional defined by 
$$
\E(\Phi^{qM}_{\la}) := \delta_{\la, \emptyset}, \quad \la\in\Y,
$$
and the inner product  $(F, G):= \E(FG)$.

\begin{theorem}\label{thm:formalqM}
Let $(a,c,d)$ be admissible in the sense of Definition \ref{meixneradmissible}. We have
$$
\left(\Phi^{qM}_{\la}, \Phi^{qM}_{\mu} \right) = \delta_{\la, \mu}h_{\la}^{qM}(q; a, c, d),
$$
where
\begin{equation}\label{eqn:hqM}
h_{\la}^{qM}(q; a, c, d) := q^{4 \{ n(\la') - n(\la) \}} \left( \frac{a^2}{c^2d^2q^3} \right)^{|\la|} 
\left(\frac{c q}{a}, \frac{d q}{a}; q\right)_{\la}.
\end{equation}
\end{theorem}

Note that $(\Phi^{qM}_{\la}, \Phi^{qM}_\la)=h_{\la}^{qM}(q; a, c, d)>0$, for any $\la\in\Y$, so that the inner product $(F, G)= \E(FG)$ in $\Sym$ is positive.

\begin{proof}
It follows from Theorem $\ref{thm:bqJqM}$ (and the definitions of the moment functionals for the big $q$-Jacobi and $q$-Meixner bases) that
$$\left(\Phi^{qM}_{\la}, \Phi^{qM}_{\mu} \right)= \lim_{b \rightarrow 0} \left(\Phi^{bqJ}_{\la}, \Phi^{bqJ}_{\mu} \right).$$
From Theorem \ref{thm:bqJorth} (ii),
\begin{equation}\label{normbqJ}
\left(\Phi^{bqJ}_{\la}, \Phi^{bqJ}_{\mu} \right) = q^{2 \{ n(\la') - n(\la) \}} 
\frac{ \left(\frac{c q}{a}, \frac{c q}{b}, \frac{d q}{a}, \frac{d q}{b}; q\right)_{\la}}{ \left( \frac{cdq^2}{ab}; q \right)_{\widehat{\la}}} \left( \frac{cd q^3}{a^2 b^2} \right)^{|\la|}.
\end{equation}
It remains to take a limit of the last expression when $b \rightarrow 0$.
One can verify $(z; q)_{\widehat{\la}} = \prod_{i=1}^{\ell(\la)}(zq^{1-2i}, zq^{2-2i}; q)_{2\la_i}$.
The part of $(\ref{normbqJ})$ that depends on $b$ can be expanded:
$$\frac{\left( \frac{cq}{b}, \frac{dq}{b} ; q \right)_{\la}}{\left( \frac{cdq^2}{ab} ; q \right)_{\widehat{\la}}} \left( \frac{1}{b^2} \right)^{|\la|} = \prod_{i = 1}^{\ell(\la)}{ \frac{(cq^{2-i}/b; q)_{\la_i}(dq^{2-i}/b; q)_{\la_i}}{(cdq^{3-2i}/ab; q)_{2\la_i}(cdq^{4-2i}/ab; q)_{2\la_i}b^{2\la_i}} }.$$
Clearly, the last expression has a limit as $b \rightarrow 0$, which is equal to
$$\prod_{i = 1}^{\ell(\la)}{ \frac{(cdq^{4-2i})^{\la_i}q^{2{\la_i \choose 2}}}{(c^2d^2q^{7-4i}/a^2)^{2\la_i}q^{2{2\la_i \choose 2}}} }
= \left( \frac{a^4}{c^3d^3} \right)^{|\la|}q^{\sum_{i=1}^{\ell(\la)}{2{\la_i \choose 2} - 2{2\la_i \choose 2} + (6i - 10)\la_i}}.$$
Therefore, back into $(\ref{normbqJ})$, we obtain
$$\lim_{b \rightarrow 0} \left(\Phi^{bqJ}_{\la}, \Phi^{bqJ}_{\mu} \right) =
q^{2\{n(\la') - n(\la)\}} \left( \frac{cq}{a}, \frac{dq}{a} ; q\right)_{\la} \left( \frac{a^2 q^3}{c^2d^2} \right)^{|\la|}q^{\sum_{i=1}^{\ell(\la)}{2{\la_i \choose 2} - 2{2\la_i \choose 2} + (6i - 10)\la_i}}$$
which can be verified to be equal to the right-hand side of $(\ref{eqn:hqM})$, by use of the identities $n(\la) = \sum_{i=1}^{\ell(\la)}{(i-1)\la_i}$ and $n(\la') = \sum_{i=1}^{\ell(\la)}{{\la_i \choose 2}}$.
\end{proof}

\subsection{Orthogonality measures}

The next theorem is stated without proof; details will be given in a separate paper. 

Fix an admissible triple $(a,c,d)$ {\rm(}Definition \ref{meixneradmissible}{\rm)}, take one of the grids   $\De_{a,\be}$ defined by \eqref{eq5.A}, and denote by $\Omega^{qM}$ the set of all point configurations on $\De_{a,\be}$ that are bounded away from $-\infty$.  Note that  symmetric functions may be evaluated at arbitrary configurations $X\in\Omega^{qM}$.

\begin{theorem}\label{thm5.A}

{\rm(i)} Let $b<0$ be an extra parameter of the form $b=\be q^M$, with $M\in\Z$.   
In the limit regime as $M \rightarrow+\infty$, the probability measures $M^{bqJ}_{q, a, b, c, d}$ defined in Theorem $\ref{thm:bqJorth}$ converge weakly to a probability measure $M^{qM}_{q, a, \be,  c, d}$ on $\Omega^{qM}$. 

{\rm(ii)} The $q$-Meixner symmetric functions $\Phi^{qM}_{\lambda} = \Phi_{\lambda}^{qM}(\ccdot; q; a, c, d)$ are square integrable with respect to $M^{qM}_{q, a, \be, c, d}$ and satisfy the orthogonality relations 
$$
\left( \Phi_{\lambda}^{qM}, \Phi_{\mu}^{qM} \right)_{L^2 \left(\Omega^{qM}, M^{qM}_{q, a, \be, c, d} \right)}
= \delta_{\lambda, \mu}h_{\la}^{qM}(q; a, c, d),
$$
where $\la,\mu\in\Y$ and $h_{\la}^{qM}(q; a, c, d)$ was defined in $(\ref{eqn:hqM})$.
\end{theorem}

Thus, given an admissible triple $(a,c,d)$, we may exhibit a family of orthogonality measures for $\{\Phi_{\lambda}^{qM}(\ccdot; q; a, c, d):\la\in\Y\}$, depending on a continuous parameter $\be$, just as in the case of univariate $q$-Meixner polynomials. 

\section{Al-Salam--Carlitz symmetric functions}\label{sect6}

We abbreviate `Al-Salam--Carlitz' by ASC.

\subsection{Univariate ASC polynomials}

There are two versions of the ASC polynomials, \cite{ASC}, \cite[3.24--3.25]{KS}, each of which transforms into the other via the involution $q \mapsto q^{-1}$.
We are interested in the polynomials that are known in the literature as ASC II polynomials.
Since the ASC I polynomials are not present in this text, we omit the II from their name and simply call them the ASC polynomials.

\smallskip

The \emph{ASC $q$-difference operator} is obtained from the $q$-Meixner operator by letting the parameter $a$ go to $0$:
$$
D^{\ASC} := \left( c - \frac{1}{x} \right) \left( d - \frac{1}{x} \right) (T_q - 1) + \frac{q}{x^2} (T_{q^{-1}} - 1).
$$
It acts on $\C[x]$, preserving its filtration by degree, as it can be seen from:
\begin{multline}\label{eqn:ASCaction}
D^{\ASC}x^n = cd(q^n - 1)x^n + \left( c - d \right)(q^n - 1)x^{n-1}\\
+ \left\{ -(q^n - 1) - q(q^{-n} - 1) \right\}x^{n-2}, \ n = 0, 1, \ldots.
\end{multline}
Since the quantities $cd(q^{n} - 1)$, $n = 0, 1, 2, \ldots$, are pairwise distinct for nonzero $c, d$, there exist monic polynomials $\varphi^{\ASC}_n = \varphi^{\ASC}_n(x; q; c, d)$ with $\deg \varphi^{\ASC}_n = n$, such that
$$
D^{\ASC}\varphi^{\ASC}_n = cd(q^n - 1)\varphi^{\ASC}_n, \ n = 0, 1, \ldots\,.
$$
The polynomials $\varphi^{\ASC}_n$ are uniquely determined by these conditions.
They are called the \emph{Al-Salam--Carlitz} polynomials.
Their hypergeometric representation is
\begin{equation}\label{ASCpolys}
\varphi^{\ASC}_n(x; q; c, d) = (-1)^n d^{-n} q^{-{n \choose 2}}
{}_2\phi_0\left[\left.\begin{matrix} q^{-n}, \; cx \\ - \end{matrix}\,\right| -\frac{d}{c}q^{n} \right], \ n = 0, 1, \dots.
\end{equation}
Our presentation was so that the following degeneration $q$-Meixner $\rightarrow$ ASC holds:
$$\lim_{a\rightarrow 0}{\varphi^{qM}_n(x; q; a, c, d)} = \varphi^{\ASC}_n(x; q; c, d).$$
In terms of the classical ASC polynomials $V_n^{(a)}(x; q)$, see \cite[3.25]{KS}, we have
\begin{equation*}
\varphi^{\ASC}_n(x; q; c, d) = (-1)^n c^{-n} V_n^{(-c/d)}(cx; q).
\end{equation*}

Like the $q$-Meixner polynomials, the ASC polynomials are related to an indeterminate moment problem and admit a variety of orthogonality measures, see \cite{C}. Below, we exhibit a particular two-parameter family of discrete orthogonality measures.

\begin{definition}\label{def:ASC}
We say that a pair $(c,d)$ of parameters is \emph{admissible} if it satisfies one of the following two conditions:
\begin{itemize}
\item $d=\bar c\in\C\setminus\R$;
\item $c$ and $d$ are real, of the same sign, and $q<cd^{-1}<q^{-1}$
\end{itemize}
(cf. Definition \ref{meixneradmissible}).
\end{definition}

Let $(\al,\be)$ be two parameters such that $\al>0$, $\be<0$; next, if $c,d>0$, then we additionally require that $\al q^m<c,d< \al q^{m-1}$, for some $m\in\Z$; likewise, if $c,d<0$, then we require that $\be q^{m-1}<c,d< \be q^m$, for some $m\in\Z$.
Given such a pair $(\al,\be)$, we consider the grid
\begin{equation}\label{eq6.A}
\De_{\al,\be} :=\be^{-1}q^\Z\sqcup \al^{-1} q^\Z= \{ \ldots, \be^{-1}q^{-1}, \be^{-1}, \be^{-1} q, \ldots \} \sqcup \{ \ldots, \al^{-1} q, \al, \al^{-1} q^{-1}, \ldots \}. 
\end{equation}
The polynomials $\varphi^{\ASC}_n(x;c,d)$ are orthogonal on $\De_{\al,\be}$ with respect to the weight function
$$
w^{\ASC}(x) := \const\frac{|x|}{(cx, dx; q)_{\infty}}, \quad  \const > 0.
$$
Note that $w^{\ASC}(x)>0$, for all $x\in\De_{\al,\be}$.

\subsection{Construction of the ASC symmetric functions}

In what follows, we assume that $(c,d)$ is admissible (Definition \ref{def:ASC}).

\begin{definition}\label{df:ASCfunction}
For any partition $\lambda$, define the \textit{Al-Salam-Carlitz symmetric function} $\Phi^{\ASC}_{\lambda}(\ccdot; q; c, d)\in\Sym$ by
\begin{equation*}
\Phi^{\ASC}_{\lambda}(X; q; c, d) := \sum_{\mu\subseteq\lambda}
\sigma^{\ASC}(\lambda, \mu; q; c, d) I^A_{\mu}(Xc; q),
\end{equation*}
where $I^A_{\mu}(\ccdot; q) \in \Sym$ are defined in Section $\ref{sec:IA}$ and, for $\mu\subseteq\lambda$,
$$
\sigma^{\ASC}(\lambda, \mu; q; c, d) := c^{-|\la|} \left( -\frac{c}{d} \right)^{|\la| - |\mu|}
q^{n(\mu') - n(\mu) - n(\lambda') + n(\lambda)}
\left[ \frac{1}{\left( q; q \right)_{\lambda_j - j - \mu_k + k}} \right]_{j, k = 1}^{\ell(\lambda)}.
$$
\end{definition}

The ASC symmetric functions are approximated by the multivariate ASC multivariate polynomials of type \eqref{eq1.B}, with the parameters varying as $N\to\infty$:$$
\varphi^{\ASC}_{\la\mid N}(x_1, \ldots, x_N; q; cq^{1-N}, dq^{1-N}) :=
\frac{\det[ \varphi^{\ASC}_{\la_i + N - i}(x_j; q; cq^{1-N}, dq^{1-N}) ]_{i, j = 1}^N}{V(x_1, \ldots, x_N)}.
$$

\begin{theorem}\label{thm:ASCpolyslimit}
For any $\lambda\in\Y$, as $N \rightarrow \infty$,
\begin{equation*}
\varphi_{\lambda | N}^{\ASC}(\ccdot; q; cq^{1-N}, dq^{1-N}) \rightarrow \Phi^{\ASC}_{\lambda}(\ccdot; q; c, d),
\end{equation*}
in the sense of Definition $\ref{def2.convergence}$.
\end{theorem}
\begin{proof}
From $(\ref{ASCpolys})$, one can verify that the polynomial $\varphi^{\ASC}_{\ell}(x; q; cq^{1-N}, dq^{1-N})$ admits the expansion
$$
\sum_{m = 0}^{\ell}{ q^{(N-1)(\ell - m)} \frac{(q; q)_{\ell}}{(q; q)_m} c^{-\ell} \left( -\frac{c}{d} \right)^{\ell - m} \frac{q^{{m \choose 2} - {\ell \choose 2}}}{(q; q)_{\ell - m}} (cx \mid q^{N-1}, q^{N-2}, q^{N-3}, \ldots)^m }.
$$
This expansion can be generalized to several variables, by the Cauchy-Binet formula, leading to
$$
\varphi^{\ASC}_{\la\mid N}(x_1, \ldots, x_N; q; cq^{1-N}, dq^{1-N}) = \sum_{\mu\subseteq\la}
{\sigma^{\ASC}_N(\la, \mu; q; c, d) I^A_{\mu\mid N}(cx_1, \ldots, cx_N; q)},
$$
where
\begin{multline*}
\sigma^{\ASC}_N(\la, \mu; q; c, d) = c^{\frac{N(N-1)}{2}} q^{(N-1)(|\la| - |\mu|)} \prod_{i=1}^N{\frac{(q; q)_{\la_i + N - i}}{(q; q)_{\mu_i + N - i}}}\left( -\frac{c}{d} \right)^{|\la| - |\mu|}\\
\times\det\left[ \frac{c^{-(\la_j + N - j)} q^{{\mu_k + N - k \choose 2} - {\la_j + N - j \choose 2}}}{(q; q)_{\la_j - j - \mu_k + k}} \right]_{j, k = 1}^{N}.
\end{multline*}
The result is then proved by checking $\sigma_N^{\ASC}(\la, \mu; q; c, d) \xrightarrow[N \rightarrow \infty]{} \sigma^{\ASC}(\la, \mu; q; c, d)$, and by using Proposition $\ref{prop2.C}$.
\end{proof}

\subsection{Limit transition: $q$-Meixner $\rightarrow$ ASC}
We recall that $(c,d)$ is assumed to be admissible (Definition \ref{def:ASC}). Let $a>0$ be an extra parameter, which will go to $0$. We want the triple $(a,c,d)$ to be admissible in the sense of Definition \ref{meixneradmissible}, so that in the case when $c,d>0$, we impose the additional constraint that $a q^m<c,d< a q^{m-1}$, for some $m\in\Z$.  

\begin{theorem}\label{thm:qMASC}
For any $\lambda\in\Y$,  
\begin{equation*}
\lim_{a\rightarrow 0}{ \Phi_{\lambda}^{qM}\left( ; q; a, c, d \right) }
= \Phi_{\lambda}^{\ASC}(\ccdot; q; c, d),
\end{equation*}
in the finite-dimensional space of symmetric functions of degree\/ $\le |\la|$.
\end{theorem}
\begin{proof}
From Definitions $\ref{df:qMfunction}$ and $\ref{df:ASCfunction}$, it suffices to show
\begin{equation}\label{eqn:qMtoASC}
\lim_{a \rightarrow 0}{\sigma^{qM}(\lambda, \mu; q; a, c, d)} = \sigma^{\ASC}(\lambda, \mu; q; c, d).
\end{equation}
We have the limit
$$\lim_{a \rightarrow 0}{a^{\la_j - j - \mu_k + k} \left( \frac{c}{a}q^{\mu_k + 2 - k}; q \right)_{\la_j - j - \mu_k + k}}
= (-cq^{\mu_k + 2 - k})^{\la_j - j - \mu_k + k}q^{{\la_j - j - \mu_k + k \choose 2}},$$
so the limit in the left-hand size of $(\ref{eqn:qMtoASC})$ exists and equals
\begin{multline*}
c^{-|\lambda|} \left( \frac{1}{qd} \right)^{|\lambda| - |\mu|}
q^{2\{n(\mu') - n(\mu) - n(\lambda') + n(\lambda)\}}\\
\times\det \left[ \frac{(-cq^{\mu_k + 2 - k})^{\la_j - j - \mu_k + k}q^{{\la_j - j - \mu_k + k \choose 2}}}
{\left( q ; q \right)_{\lambda_j - j - \mu_k + k}} \right]_{j, k = 1}^{\ell(\lambda)}
\end{multline*}
\begin{multline}\label{eqn:limittoASC}
=c^{-|\lambda|} \left( - \frac{c}{d} \right)^{|\lambda| - |\mu|}
q^{2\{n(\mu') - n(\mu) - n(\lambda') + n(\lambda)\}}\\
\times\det \left[ \frac{(q^{\mu_k + 1 - k})^{\la_j - j - \mu_k + k}q^{{\la_j - j - \mu_k + k \choose 2}}}
{\left( q ; q \right)_{\lambda_j - j - \mu_k + k}} \right]_{j, k = 1}^{\ell(\lambda)}.
\end{multline}
The $(j, k)$ entry of the matrix above is a fraction where the numerator is $q$ raised to the power of
\begin{gather*}
(\mu_k+1-k)(\la_j - j - \mu_k + k) + {\la_j - j - \mu_k + k \choose 2} =\\
\frac{(\la_j - j - \mu_k + k)(\la_j - j + \mu_k - k + 1)}{2} = \frac{(\la_j - j  + \frac{1}{2})^2 - (\mu_k - k + \frac{1}{2})^2}{2}.
\end{gather*}
We can calculate (similar calculations were done many times by now)
\begin{gather*}
\sum_{j = 1}^{\ell(\la)}{\frac{(\la_j - j + 1/2)^2}{2}} - \sum_{k = 1}^{\ell(\la)}{\frac{(\mu_k - k + 1/2)^2}{2}}
= n(\la') - n(\la) - n(\mu') + n(\mu).
\end{gather*}
It follows that $(\ref{eqn:limittoASC})$ equals $\sigma^{\ASC}(\la, \mu; q; c, d)$, as desired.
\end{proof}

\subsection{Formal orthogonality}

As usual, because $\{ \Phi^{\ASC}_{\la} := \Phi^{\ASC}_{\la}(\cdot; q; c, d)\}_{\la\in\Y}$ is a basis of $\Sym$, we can introduce the moment functional $\E$ by setting 
$$
\E(\Phi^{\ASC}_{\la}) := \delta_{\la, \emptyset}, \quad \la\in\Y,
$$
and define the corresponding inner product  by $(F, G):= \E(FG)$.

\begin{theorem} Let $(c,d)$ be admissible {\rm(}Definition \ref{def:ASC}{\rm)}. We have
$$
\left(\Phi^{\ASC}_{\la}, \Phi^{\ASC}_{\mu} \right)= \delta_{\la, \mu}h_{\la}^{\ASC}(q; c, d), \quad \la,\mu\in\Y,
$$
where
\begin{equation}\label{eqn:hASC}
h_{\la}^{\ASC}(q; c, d) := \frac{q^{2\{ n(\la) - n(\la') \}}}{(cdq)^{|\la|}}.
\end{equation}
\end{theorem}

Note that $h_{\la}^{\ASC}(q; c, d)>0$, for all $\la\in\Y$.

\begin{proof}
The proof is similar to that of Theorem $\ref{thm:formalqM}$.
The only task is to calculate the limit of $(\ref{eqn:hqM})$ when $a \rightarrow 0$.
This is a result of the following calculation:
\begin{gather*}
\lim_{a \rightarrow 0}{a^{2|\la|}\left( \frac{cq}{a}, \frac{dq}{a}; q \right)_{\la}}
= \lim_{a \rightarrow 0}{\prod_{i=1}^{\ell(\la)}{ a^{2\la_i} \left( \frac{cq^{2-i}}{a}, \frac{dq^{2-i}}{a}; q \right)_{\la_i} }}
= \prod_{i=1}^{\ell(\la)}{(cdq^{4-2i})^{\la_i}q^{2{\la_i \choose 2}}}\\
= (cdq^2)^{|\la|} q^{2\sum_{i=1}^{\ell(\la)}{-(i - 1)\la_i + {\la_i \choose 2}}}
= (cdq^2)^{|\la|} q^{2\{ -n(\la)+n(\la') \}}.
\end{gather*}
\end{proof}

\subsection{Orthogonality measures}

The theorem below is similar to Theorem \ref{thm5.A}; its proof will be given in a separate paper. 

Fix an admissible pair $(c,d)$ {\rm(}Definition \ref{def:ASC}{\rm)}, take one of the grids   $\De_{\al,\be}$ defined by \eqref{eq6.A}, and denote by $\Omega^{\ASC}$ the set of all point configurations on $\De_{\al,\be}$ that are bounded away from $\pm\infty$.  Note that  symmetric functions may be evaluated at arbitrary point configurations $X\in\Omega^{qM}$.

\begin{theorem}\label{thm6.A}

{\rm(i)} Let $a>0$  be an extra parameter of the form $a=\al q^M$, with $M\in\Z$.   
In the limit regime as $M \rightarrow+\infty$, the probability measures $M^{qM}_{q, a, \be, c, d}$ defined in Theorem $\ref{thm5.A}$ converge weakly to a probability measure $M^{ASC}_{q, \al, \be,  c, d}$ on $\Omega^{\ASC}$. 

{\rm(ii)} The ASC symmetric functions $\Phi^{\ASC}_{\lambda} = \Phi_{\lambda}^{\ASC}(\ccdot; q;c, d)$ are square integrable with respect to $M^{ASC}_{q, \al, \be, c, d}$, and satisfy the orthogonality relations 
$$
\left( \Phi_{\lambda}^{\ASC}, \Phi_{\mu}^{\ASC} \right)_{L^2 \left(\Omega^{\ASC}, M^{\ASC}_{q, \al, \be, c, d} \right)}
= \delta_{\lambda, \mu}h_{\la}^{\ASC}(q; a, c, d),
$$
where $\la,\mu\in\Y$, and $h_{\la}^{\ASC}(q; c, d)$ was defined in $(\ref{eqn:hASC})$.
\end{theorem}

Thus, given an admissible pair $(c,d)$, we may exhibit a family of orthogonal measures for $\{\Phi_{\lambda}^{\ASC}(\ccdot; q; c, d):\la\in\Y\}$ depending on two continuous parameters $\al$ and $\be$, just as in the case of univariate ASC polynomials.

\begin{theorem}\label{thm6.B}
Fix an admissible pair $(c,d)$ {\rm(}Definition \ref{def:ASC}\,{\rm)}, and let $\al>0$ and $\be<0$  be arbitrary.  Consider the the grid $\De_{\al, \be}$ defined in \eqref{eq6.A} and denote by $\Omega^{\ASC}$ the set of all point configurations on $\De_{\al, \be}$ that are bounded away from $-\infty$ and $+\infty$.

{\rm(i)} Let $a$ range over the grid $\{\al^{-1} q^M: M\in\Z\}$.  
In the limit regime as $M \rightarrow+\infty$, the probability measures $M^{qM}_{q, a, \be, c, d}$ defined in Theorem \ref{thm5.A} converge weakly to a probability measure $M^{\ASC}_{q, \al, \be, c, d}$ on $\Omega^{\ASC}$.

{\rm(ii)} The $ASC$ symmetric functions $\Phi^{\ASC}_{\lambda} = \Phi_{\lambda}^{\ASC}(\ccdot; q; c, d)$ are square integrable with respect to $M^{\ASC}_{q, \al, \be, c, d}$ and satisfy the orthogonality relations 
$$
\left( \Phi_{\lambda}^{\ASC}, \Phi_{\mu}^{\ASC} \right)_{L^2 \left(\Omega^{\ASC}, M^{\ASC}_{q, \al, \be, c, d} \right)}
= \delta_{\lambda, \mu}h_{\la}^{\ASC}(q; c, d),
$$
where $\la,\mu\in\Y$, and $h_{\la}^{\ASC}(q; c, d)$ was defined in \eqref{eqn:hASC}.
\end{theorem}

\bigskip

Cesar Cuenca: Department of Mathematics, MIT, Cambridge, MA, USA.

Email address: cuenca@mit.edu

\bigskip

Grigori Olshanski: Institute for Information Transmission Problems, Moscow, Russia; Skolkovo Institute of Science and Technology, Moscow, Russia; National Research University Higher School of Economics, Moscow, Russia.

Email address: olsh2007@gmail.com

\end{document}